\documentclass[12pt, reqno]{amsart}
\usepackage{amsfonts}
\usepackage{bbm}
\usepackage{amscd,amsfonts}
\usepackage{amssymb, eucal, amsfonts, amsmath, xypic, latexsym, tikz}
\usepackage{pifont}
\usepackage{mathrsfs,color}
\usepackage{amsthm,indentfirst,bm,fancyhdr,dsfont}
\usepackage{graphicx}
\usepackage[all]{xy}
\usepackage[CJKbookmarks=true]{hyperref}
\usepackage{extarrows}

\usepackage{mathrsfs}
\usepackage{amsmath}
\usepackage{amssymb}
\usepackage{hyperref}

\include{micro}

\newtheorem{thm}{Theorem}[section]
\newtheorem{lemma}[thm]{Lemma}
\newtheorem{remark}[thm]{Remark}
\newtheorem{corollary}[thm]{Corollary}

\theoremstyle{definition}

\newtheorem{proposition}[thm]{Proposition}
\newtheorem{example}[thm]{Example}

 \theoremstyle{remark}

\numberwithin{equation}{section}
\newcommand{\abs}[1]{\lvert#1\rvert}

\begin{document}

\def\frakl{{\mathfrak L}}
\def\frakg{{\mathfrak G}}
\def\bbf{{\mathbb F}}
\def\bbl{{\mathbb L}}
\def\bbz{{\mathbb Z}}
\def\bbr{{\mathbb R}}
\def\bbc{{\mathbb C}}

\def\bvp{\bf{\varphi}}

\def\ad{\textsf{ad}}
\def\GL{\text{GL}}
\def\Der{\mbox{\rm Der}}
\def\Hom{\textsf{Hom}}
\def\ind{\textsf{ind}}
\def\res{\textsf{res}}
\def\gl{\frak{gl}}
\def\sl{\frak{sl}}
\def\Ker{\textsf{Ker}}
\def\Lie{\textsf{Lie}}
\def\id{\textsf{id}}
\def\det{\textsf{det}}
\def\Lie{\textsf{Lie}}
\def\Aut{\textsf{Aut}}
\def\Ext{\textsf{Ext}}
\def\Coker{\textsf{{Coker}}}
\def\dim{\textsf{{dim}}}
\def\pr{\mbox{\sf pr}}
\def\SL{\text{SL}}

\def\geqs{\geqslant}

\def\ba{{\mathbf a}}
\def\bd{{\mathbf d}}
\def\bbk{{\mathbb K}}
\def\co{{\mathcal O}}
\def\cn{{\mathcal N}}
\def\cv{{\mathcal V}}
\def\cz{{\mathcal Z}}
\def\cq{{\mathcal Q}}
\def\cf{{\mathcal F}}
\def\cc{{\mathcal C}}
\def\ca{{\mathcal A}}

\def\sl{{\frak{sl}}}

\def\ggg{{\frak g}}
\def\lll{{\frak l}}
\def\hhh{{\frak h}}
\def\nnn{{\frak n}}
\def\sss{{\frak s}}
\def\bbb{{\frak b}}
\def\ccc{{\frak c}}
\def\ooo{{\mathfrak o}}
\def\ppp{{\mathfrak p}}
\def\uuu{{\mathfrak u}}

\def\p{{[p]}}
\def\modf{\text{{\bf mod}$^F$-}}
\def\modr{\text{{\bf mod}$^r$-}}

\title[Irreducible modules of $\sl_{mp}$ in characteristic $p$] {Irreducible modules of $\frak {sl}_{mp}$ in characteristic $p$ with regular or subregular nilpotent $p$-characters}
\author{Bin Liu, Bin Shu and Xin Wen}
\address{School of Mathematical Sciences, East China Normal University, Shanghai, 200241, China.} \email{1918724868@qq.com}
\address{School of Mathematical Sciences, East China Normal University, Shanghai, 200241, China.} \email{bshu@math.ecnu.edu.cn}
\address{Ping An Technology, Shanghai, Co., Ltd, 200120.}\email{dios0000@163.com}
\subjclass[]{}
 \keywords{regular nilpotent $p$-characters, subregular nilpotent $p$-characters, standard Levi forms}
\thanks{This work is partially supported by the NSF of China (Grants: 12071136 and 11771279) and Shanghai Key
Laboratory of PMMP (No. 13dz2260400).}

\begin{abstract} Let $\bbk$ be an algebraically closed field of prime characteristic $p$. If $p$ does not divide $n$, irreducible modules over $\frak {sl}_n$ for regular and subregular nilpotent representations have already known(see \cite{Jan2} and \cite{Jan3}). In this article, we investigate the question when $p$ divides $n$, and precisely describe simple modules of $\frak {sl}_n$ for regular and subregular nilpotent representations.
\end{abstract}

\maketitle
\section*{0. Introduction}

\subsection{A general setting-up} Let $\bbk$ be an algebraically closed field of characteristic $p>0$. We denote by $\frak {gl}_n:=\gl_n(\bbk)$ the general linear Lie algebra and by $\sl_n:=\sl_n(\bbk)$ the special linear Lie algebra over $\bbk$. Let $\ggg=\gl_n$ or $\sl_n$ in this section. Obviously, $\ggg$ is a restriceted Lie algebra with $p$-th power map, denote it by $x \mapsto x^{[p]}$.  Associated with any given $p$-character $\chi\in \ggg^*$, the $\chi$-reduced enveloping algebra
$U_{\chi}(\ggg)$ is defined to be the quotient of the universal enveloping algebra $U(\ggg)$ by the ideal generated by all $x^p-x^{[p]}-\chi(x)^p$, $x\in \ggg$. Each isomorphism class of irreducible representations of $\ggg$ is associated with a unique $p$-character $\chi$. The corresponding irreducible representations are called $\chi$-reduced irreducible representations, which are totally irreducible representations over the $\chi$-reduced enveloping algebra $U_\chi(\ggg)$. More generally,  one can regard the category of $U_\chi(\ggg)$-modules as a subcategory of $\ggg$-module category associated with a given $p$-character $\chi$, denote by $\cc_\chi$.
Let $\hhh$ be the Cartan subalgebra of $\ggg$ consisting of diagonal matrices in $\ggg$. Then $\hhh$ is spanned by $\{E_{ii}\mid i=1,\ldots,n\}$ for $\ggg=\gl_n$ where $E_{ii}$ is an elementary matrix with $(i,i)$-entry being equal to $1$, and all other entries equal to zero.
Denote by $\varepsilon_i\in\hhh^*$ defined via mapping a matrix in $\frak h$ to its $i$-th diagonal entry. Then $\ggg$ has a simple root system $\Pi:=\{\alpha_i=\varepsilon_i-\varepsilon_{i+1}\mid i=1,\ldots,n-1\}$, and the root system $R=\{\varepsilon_i-\varepsilon_j\mid  1\leq i\ne j\leq n \}$.  The system $R$  is divided into the positive part $R^+=\{ \varepsilon_i-\varepsilon_j\mid  1\leq i< j\leq n \}$ and the negative part $R^-=-R^+$.  We have a triangular decomposition $\ggg=\nnn^+ \oplus \frak h \oplus \frak n^-$ with $\nnn^\pm=\sum_{\alpha\in R^\pm}\ggg_\alpha$, and $\ggg_\alpha=\{X\in\ggg\mid \ad H(X)=\alpha(H)X\;\; \forall H\in\hhh\}$. Let $\bbb$ be the Borel subalgebra corresponding to $R^+$, which is equal to $\hhh\oplus \nnn^+$.
Let $W$ be the Weyl group of $\ggg$ which is equal to the symmetric group $\frak S_n$ with $n$ letters. Let $\varpi_1, \cdots, \varpi_{n-1}$ be the dual basis of $\Pi$, namely $\left \langle\varpi_i,{\alpha_{j}}^{\vee} \right \rangle=\delta_{ij}$ with $1 \leq i,j \leq n-1$. We call them the fundamental dominant weights. Set $\rho=\dfrac{1}{2}\sum_{\alpha \in R^+}\alpha$, it is easy to see $\rho=\sum_{i=1}^{n-1}\varpi_i$.


\subsection{Nilpotent $p$-characters} In this subsection, denote $G=\SL_n$ the special linear group over $\bbk$ and denote $\ggg=\sl_n$, the Lie algebra of $G$.
In the present case $p\mid n$,  there is no longer a non-degenerate $G$-equivariant bilinear form on $\ggg$. Correspondingly, there is no longer a $G$-equivariant isomorphism
between $\ggg$ and $\ggg^*$. So it is not available to carry Jordan-Chevalley decomposition from $\ggg$ to $\ggg^*$. Fortunately, Kac-Weisfeiler  established “Jordan-Chevalley decomposition" directly on $\ggg^*$ in \cite{KW2}.  Set
\begin{align*}
&(\nnn^-)^*=\{\phi\in \ggg^*\mid \phi(\hhh+\nnn^+)=0  \}, \cr
&(\hhh)^*=\{\phi\in\ggg^*\mid \phi(\nnn^++\nnn^-)=0\},\cr
&(\bbb^-)^*=\{\phi\in \ggg^*\mid \phi(\nnn^+)=0  \}.
\end{align*}
By \cite[Theorem 4]{KW2}, any $\phi\in \ggg^*$ has a unique pair of linear functions $\phi_s$ and $\phi_n$ such that
``Jordan decomposition" $\phi=\phi_s+\phi_n$ holds, satisfying the conditions
 \begin{itemize} \label{Jordan dec}
 \item[(1)] $g.\phi_s\in (\hhh)^*$, and  $g.\phi_n\in (\nnn^-)^*$ for some $g \in G$;
   \item[(2)] If $g.\phi_s(H_\alpha)\ne 0$, then $g.\phi_n(X_\alpha)=g.\phi_n(X_{-\alpha})=0$ for $\phi\in R^+$, where $X_\alpha=E_{ij}$ for $\alpha=\varepsilon_i-\varepsilon_j$, and $H_\alpha=E_{ii}-E_{jj}$.
        \end{itemize}
  So, any $\phi\in\ggg^*$ lies in the coadjoint $G$-orbit of $(\bbb^-)^*$. Up to the coadjoint $G$-conjugation, the centralizer of $\phi_s$ in $\ggg$ is a Levi subalgebra associated with the root subsystem $\{\alpha\in R\mid g.\chi_s(H_\alpha)=0 \}$.  For any $\chi\in\ggg^*$, $\chi$ is called nilpotent (or coadjoint nilpotent) if $\chi=\chi_n$. All of those nilpotent $\chi$ constitute a $G$-stable variety. This variety will be called the coadjoint nilpotent cone of $\ggg^*$,  denoted by $\mathscr{N}_{\ggg^*}$. Note that if $\chi_1$ and $\chi_2$ lie in some $G$-orbit, then $U_{\chi_1}(\ggg)$ and $U_{\chi_2}(\ggg)$ are isomorphic as algebras because any element of $G$ gives rise to an restricted automorphism of $\sl_n$. So we are only concerned with a representative of the coadjoint $G$-orbit of $\chi$ when we investigate the representations of $U_\chi(\ggg)$.

\subsubsection{Kac-Weisfeiler's Morita equivalence in the case $\sl_n$ with $p\mid n$} For any $\chi\in \ggg^*$ for $\ggg=\sl_n$, suppose $\chi=\chi_s+\chi_n$ is Jordan decomposition as in the previous paragraph.
Modulo the coadjoint $G$-conjugation, the centralizer $\ggg_{\chi_s}$  of $\chi_s$ in $\ggg$ is a Levi subalgebra associated with the root subsystem $\{\alpha\in R\mid \chi_s(H_\alpha)=0 \}$. In \cite[Theorem 2]{KW} (and formuated by Friedlander-Parshall in \cite[Theorems 3.2 and 8.3]{FP1}), Kac-Weisfeiler' Morita equivalence indicates that  the category of $U_\chi(\ggg)$-modules is equivalent to  that of $U_{\chi}(\ggg_{\chi_s})$.
The latter is equivalent to the category of $U_{\chi_n}(\ggg_{\chi_s})$-modules, which is suitable to the case of $\sl_n$ when $p\mid n$.
This means that representations of $U_\chi(\ggg)$ for any $\chi\in\ggg^*$ can be reduced to the case when $\chi$ is nilpotent, up to Morita equivalence.

 \subsubsection{}\label{X_p} From now on, we are then concerned with the situation that   $\chi$ is nilpotent, satisfying $\chi(\bbb)=0$. In this case, $U_\chi(\bbb)=U_0(\bbb)$, any irreducible $U_0(\bbb)$-module is one-dimensional with $\nnn^+$-trivial action, and $\hhh$-scale action by  $\lambda\in \mathscr{X}_p \subsetneq \hhh^*$ which by definition consists of  all linear function on $\hhh$  satisfying  $\lambda(H)^p=\lambda(H)$ for all $H\in \hhh$. The corresponding one-dimensional $U_0(\bbb)$-module is denoted by $\bbk_\lambda$.
Define a $U_\chi(\ggg)$-module  induced from $U_{\chi}(\bbb)$-module $\bbk_\lambda$
\begin{align}\label{eq: baby verma}
Z_\chi(\lambda)=U_\chi(\ggg)\otimes_{U_\chi(\bbb)}\bbk_\lambda
\end{align}
which is called a baby Verma module.

Note that if $\ggg=\frak {gl}_n$, then $\mathscr X_p=\bigoplus_{i=1}^{n}\mathbb{F}_p\varepsilon_i$; if $\ggg=\frak {sl}_n$, then $\mathscr X_p=\bigoplus_{i=1}^{n-1}\mathbb{F}_p\alpha_i=\{x \mid x=\sum_{i=1}^{n}a_i \varepsilon_{i},\text{with }a_i \in \mathbb{F}_p \text{ and }\sum_{i=1}^{n}a_i=0\}$.

\subsection{Nilpotent $p$-characters of standard Levi-forms}

\subsubsection{} By exploiting the grading structure on $\ggg^*$ of a co-character of $G$ (see \cite{L11}), Lusztig investigated the coadjoint nilpotent cone in arbitrary prime characteristic (see \cite[Theorem 2.2]{L10}). Applying his result to $\SL_n$ with $p\mid n$, one has that $\mathscr{N}_{\ggg^*}$ can be divided into  locally closed $G$-invariant smooth subvarieties indexed by the unipotent classes in $\SL_n(\bbc)$ over the complex numbers.

 So we can say that a $p$-character $\chi\in\mathscr{N}_{\ggg^*}$ is  regular (resp. subregular) if under Lusztig's partition of $\mathscr{N}_{\ggg^*}$, its coadjoint $G$-orbit corresponds to the regular nilpotent orbit (resp. the subregular nilpotent orbit) over the complex numbers.

In particular, we take the standard representatives of the regular nilpotent orbit and of the subregular nilpotent orbit as below. For the  former, it is defined by
 \begin{equation}\label{eq: regular}
 \chi(X_{-\alpha_i})=1, i=1,\ldots,n-1\mbox{ while }\chi(\hhh \oplus \ggg_\alpha)=0 \mbox{ for all }\alpha\in R\backslash (-I),
 \end{equation}
 with $I= \{\alpha_1,\alpha_2,\ldots,\alpha_{n-1}\}$.
  And for the latter, it is defined by
  \begin{equation}\label{eq: subregular}
  \chi(X_{-\alpha_i})=1, i=1,\ldots,n-2 \mbox{ while } \chi(\hhh \oplus \ggg_\alpha)=0 \mbox{ for all }
  \alpha\in R\backslash (-I)
  \end{equation}
with $I=\{\alpha_1,\ldots, \alpha_{n-2}\}$.  In Friedlander-Parshall's term, those standard representatives are of standard Levi-form associated with $I= \{\alpha_1,\alpha_2,\ldots,\alpha_{n-1}\}$, and $I=\{\alpha_1,\ldots, \alpha_{n-2}\}$, respectively.

Let $s_{\alpha}$ with $\alpha \in R$ be the reflection, i.e. $s_{\alpha}(x)=x-\left \langle x,\alpha^{\vee} \right \rangle \alpha$ where $x \in \ggg^*$. Obviously, $s_{\varepsilon_i-\varepsilon_j}$($1 \leq i \neq j \leq n-1$) is identified with transposition $(i,j)$. In particular, $s_{\alpha_{i}}$($1 \leq i \leq n-1$) is identified with transposition $(i,i+1)$. Since $W$ be the Weyl group of $\ggg$, it is clear that $W$ is generated by all reflection $s_{\alpha}$ with $\alpha \in R$, i.e. $W$ is generated by all reflection $s_{\alpha}$ with $\alpha \in \Pi$. We can identify $W$ with the symmetic group $\frak S_n$. For $w \in W$ and $\lambda \in \frak h^*$, define a dot action of $W$ by $w.\lambda=w(\lambda+\rho)-\rho$.

We denote by $W_I$ the subgroup of $W$ generated by all reflection $s_{\alpha}$ with $\alpha \in I$. Obviously, when $I=\{\alpha_1,\alpha_2,\ldots,\alpha_{n-1}\}$, $W_I$ coincides with the whole Weyl group $W$ of $\sl_n$ which can be identified with $\frak S_{n}$.

 \subsubsection{} If $\chi$ is a nilpotent element of standard Levi form, then by Friedlander-Parshall and Jantzen's results (\cite{FP1}, \cite{Jan1} and \cite[Lemma D.1]{Jan2}), we have

 \begin{lemma}\label{lem: univ} Any irreducible $U_\chi(\ggg)$-module must be isomorphic to a unique simple quotient of $Z_\chi(\lambda)$ for some $\lambda\in  \mathscr{X}_p$, which is denoted by $L_\chi(\lambda)$.
\end{lemma}

\subsubsection{} In this paper, we study the irreducible representations of $U_\chi(\ggg)$ for $\ggg=\sl_n$, $p\mid n$ and $\chi$ is of the form as in \eqref{eq: regular} or in \eqref{eq: subregular}, where $W$ is the Weyl group of $\ggg$.

The following facts are fundamental and very important to the sequel.

\begin{lemma} (\cite[Propositions C.2 and D.3]{Jan2}) \label{lem: fundamental iso} Suppose $\chi \in \ggg^*$ is a nilpotent element of standard Levi form associated with $I$. If $\lambda,\mu\in \mathscr{X}_p$ with $\mu \in W_I.\lambda$, then  we have $U_\chi(\ggg)$-module homomorphisms $Z_\chi(\lambda)\cong Z_\chi(\mu)$ and $L_\chi(\lambda)\cong L_\chi(\mu)$, where $W$ is the Weyl group of $\ggg$.
\end{lemma}

\begin{lemma}  \label{lem: fundamental iso 2} Suppose $\ggg=\frak{gl}_n$ and $\chi \in \frak{gl}^*_n$ is a nilpotent element of standard Levi form associated with $I$. Then for $\lambda,\mu\in \mathscr{X}_p$, $L_\chi(\lambda)\cong L_\chi(\mu)$ if and only if $\mu \in W_I.\lambda$, where $W$ is the Weyl group of $\ggg$.
\end{lemma}



\subsection{} In the case $\ggg=\sl_n$ with $p\nmid n$ or $\ggg=\gl_n$ with $p$ being any prime number, there exists a non-degenerate bilinear $G$-invariant form on $\ggg$ for $G=\text{GL}_n$ in characteristic $p$. In this case, $\chi$-reduced irreducible modules of $\ggg$ can be  precisely described for $\chi$ being regular nilpotent (see \cite{FP1}) and for $\chi$ being subregular nilpotent (see \cite{Jan3}). However, when $p\mid n$ the arguments in \cite{FP1} and in \cite{Jan3} are not valid ever for $\ggg=\sl_n$. Actually, it has been an open problem how to parameterize the isomorphism classes of irreducibles and how to construct the irreducible modules for $\ggg=\sl_n$ with $p\mid n$.

In this note, we investigate this problem when $\chi$ is regular nilpotent or subregular nilpotent.
In this special case, taking aid of the known results on $\gl_n$ and $\sl_{n+1}$ with $p \mid n$ and there by $p\nmid n+1$,  we finally determine the iso-classes of irreducible modules and construct the irreducible modules for $\sl_n$ with $p\mid n$  when $\chi$ is regular or subregular. In the following, we introduce  the main results and the structure of the present paper.

\subsection{}\label{FP} One source of this paper is partially from the third author's PhD thesis in 2016 (see \cite{Wen}), which was to describe  composition factors of baby Verma modules for $\sl_n$ when $p\mid n$ and $p$-character $\chi$ is regular nilpotent. Recall that when $p\nmid n$, it is a standard result that $Z_\chi(\lambda)=L_\chi(\lambda)$ (see \cite{FP1} and \cite{Jan1}) for $\ggg=\sl_n$ or $\ggg=\gl_n$. However, when $p\mid n$, it remains true only for $\gl_n$ because in this case there remains a non-degenerate bilinear form on $\gl_n$ while there is no longer non-degenerate bilinear form on $\sl_n$.

The main idea of deciding the isomorphism classes of irreducible $\sl_n$-module in this case is to analysis when $\res_{\chi}Z(\lambda)$ is irreducible (the restriction to $\sl_n$), and what the composition series we have in $\res_{\chi}Z(\lambda)$. Note that as a corollary to Lemma \ref{lem: fundamental iso}, $Z_\chi(\lambda)\cong Z_\chi(w.\lambda)$ for any $w\in W$  when $\chi$ is regular of form \eqref{eq: regular}.

Summing up, it can be understood  via the action of the Weyl group as below.
\begin{thm}\label{thm: 2.4.12} Suppose that $\chi\in \frak {sl}_n^*$ is regular nilpotent of form \eqref{eq: regular} and $\lambda \in \mathscr{X}_p$, then $Z_\chi(\lambda)$ is an $U_\chi(\sl_n)$-module.
\begin{itemize}
\item[(1)] If $\lambda \in W.0$, then $Z_\chi(\lambda)$ is reducible, admitting a composition series of length $p$. All composition factors are isomorphic to its unique simple quotient $L_\chi(\lambda)$.
\item[(2)] If $\lambda\notin W.0$, then  $Z_\chi(\lambda)$ itself is an irreducible $U_\chi(\frak {sl}_n)$-module.
    \end{itemize}
\end{thm}

The proof of the above theorem will be given in \S\ref{sec: proof of thm 0.2}.  It is worth mentioning that Jantzen independently dealt it by different method the restriction from $\gl_n$ to $\sl_n$ in  generality  (see \cite{Jan4}).

\subsection{} The following theorem further indicates the construction of irreducible module $L_\chi(\lambda)$ when $\lambda\in W.0$  and $\chi\in \frak {sl}_n^*$ is regular nilpotent of form \eqref{eq: regular}.

\begin{thm}\label{Theorem 2.5.5} Let $\ggg=\frak {sl}_{n}$ with $p\mid n$ and $\widetilde{\ggg}=\frak {sl}_{n+1}$. Suppose $\chi\in\ggg^*$ is regular of the form \eqref{eq: regular}. Set $\theta_i=\sum_{k=i}^n\alpha_k$ for $k=1,\ldots,n$, and $u=\prod_{i=1}^{n}X_{-\theta_i}^{p-1}\in U_{\chi}(\widetilde\ggg)$ where $\chi$ is naturally regarded an element of ${\widetilde\ggg}^*$. Then
\begin{itemize}
\item[(1)] $L_\chi(0) \cong u \widetilde{L_\chi}(0)$ as $U_{\chi}(\ggg)$-modules, where $\widetilde{L_\chi}(0)$ is the irreducible module of weight $0$ over $U_\chi(\widetilde\ggg)$ which is precisely described.
\item[(2)] $L_\chi(\lambda)\cong L_\chi(0)$ as $U_{\chi}(\ggg)$-modules for all $\lambda\in W.0$, where $W$ is the Weyl group of $\ggg$.
\end{itemize}
\end{thm}
The proof will be given in \S\ref{sec: proof of thm 0.3}.

\subsection{} For the subregular nilpotent cases, there is a basic  observation that $\gl_{n}$ in the case when $p\mid n$ satisfies the existence condition of a non-degenerate bilinear form. This is a start-point of our investigation. Combining with Jantzen's arguments in \cite{Jan4}, we have the following result.

\begin{thm}\label{thm:subregular} Assume that $\chi\in \frak {sl}_n^*$ is a subregular nilpotent $p$-character of form \eqref{eq: subregular}.  Then for any $\lambda\in \mathscr{X}_p$, there are  correspondingly $\lambda_i \in \mathscr{X}_p$ with $i=0,1,\ldots,n-1$, such that for all $0 \leq i \leq n-1$ the following two statements hold:
\begin{itemize}

 \item[(1)]    $L_\chi(\lambda_{i})$ are irreducible for all $\lambda_i$, spanned by a basis in \eqref{eq: subregular basis}. It satisfies $\text{dim}L_\chi(\lambda_{i})=p^{N-1}r_{n-1-i}$ (where $N=\abs{R^+}$) and $L_\chi(\lambda_{i}) \cong Z_\chi(\lambda_{i},\alpha_{n-1})$ as $U_{{\chi}}(\frak {sl}_n)$-modules.
\item[(2)] $Z_\chi(\lambda_{i})$ just has composition factors of the form $L_\chi(\lambda_{j})(0 \leq j \leq n-1)$ satisfying $[Z_\chi(\lambda_{i}):L_\chi(\lambda_{j})]=N_j$ where $N_j$ is defined in Corollary \ref{cor: 3.1.5}.
         \end{itemize}
\end{thm}

The proof will be given in \S\ref{sec: proof of thm 0.5}.

\subsection{} For general theory and recent advances on modular representations of reductive Lie algebras, one can refer to \cite{BM}-\cite{Wil} and \cite{Jan1}-\cite{Jan2} {\textsl{etc}.}.

\section{Preliminaries}

Keep the notations as in the introduction. In particular, let $\ggg=\gl_n$ or $\sl_n$. The simple roots of $\ggg$ consists of $\alpha_i=\varepsilon_i-\varepsilon_{i+1}$ with $1 \leq i \leq n-1$.

\subsection{$X(T)$ and $Y(T)$} Let $T$ be the maximal torus in $\GL_n(\bbk)$ consisting of diagonal matrices. Denote by $X(T)$  the character group of $T$, and by $Y(T)$ its dual group which is an one parameter multiplicative subgroup(\cite[16.1]{Hum 2}). Denote the standard paring of $X(T)$ and $Y(T)$ by $\langle \cdot,\cdot\rangle$. This means that if $\lambda\in X(T)$ and $\phi\in Y(T)$, then $\lambda(\phi(a))=a^{\langle \lambda, \phi\rangle}$ for $a\in \bbk^{\times}:=\bbk\backslash \{0\}$.
 Suppose that $W$ is the Weyl group of $\GL_n(\bbk)$ which can be identified with the symmetric group.

 Clearly, $\varepsilon_i \in X(T)$ that maps a matrix in $T$ to its $i$-th diagonal entry.  The root system $R$ mentioned in the introduction also becomes  the root system  of   $\GL_n(\bbk)$ with respect to $T$, which is a subset of $X(T)$.  For any $\alpha\in R$, denote the coroot by $\alpha^\vee\in Y(T)$ (see \cite[\S7.4]{Sp}).

 The fundamental weights will be denoted by $\varpi_i$ with $i=1,\ldots, n-1$.
 Denote by $W$ the Weyl group of $\frak{gl}_n$.

\subsection{The first dominant alcove}
Let $C=\{\lambda\in X(T) \mid 0\leq\langle\lambda+\rho,\alpha^{\vee}\rangle \leq p \mbox{ for all }\alpha\in R^+ \}$ be the closure of the first fundamental alcove of $X(T)$.

\subsection{Some additional conventions}\label{SAC}  For a simple root $\alpha$, we set $\frak p_\alpha=\frak b \oplus {\ggg_{-\alpha}}$ which is a parabolic subalgebra associated with the subsystem $\{\alpha\}$. Let $\bbk_\lambda$ is a one dimensional $U_0(\frak b)$-module with $\frak h$ acts as multiple.
For $\lambda \in X(T)$, if $\ggg=\gl_{n}$, it is clear that $d\lambda \in \Lambda:=X(T)\slash pX(T)$ which is equal to $\oplus_{i=1}^n \mathbb{F}_p \varepsilon_{i}$, i.e. $\mathscr X_p=\Lambda$.  For convenience, we still use $\lambda$ instead of $d\lambda$ if the context is clear.
We set $Z_{\chi,\alpha}(\lambda)=U_\chi(\frak p_\alpha) \otimes_{U_{0}(\frak b)}\bbk_\lambda$.

Let $\langle\lambda+\rho,\alpha^{\vee}\rangle \equiv m$(mod $p$) with $0<m \leq p$, define $\varphi:Z_{\chi,\alpha}(\lambda-m\alpha) \rightarrow Z_{\chi,\alpha}(\lambda)$, $\varphi(X_{-\alpha}^j \otimes 1)=X_{-\alpha}^{m+j} \otimes 1$. If  $\chi(X_{-\alpha})=0$, then $\varphi$ is a homomorphism with $\text{Im}\varphi \neq Z_{\chi,\alpha}(\lambda)$. The quotient $Z_{\chi,\alpha}(\lambda)/\text{Im}\varphi$ is a simple $\frak p_\alpha$-module which is denoted by $L_{\chi,\alpha}(\lambda)$. Define $U_\chi(\ggg)$-module
\begin{align}\label{eq: levi verma}
Z_\chi(\lambda,\alpha)=U_\chi(\ggg)\otimes_{U_{0}(\frak p_\alpha)}L_{\chi,\alpha}(\lambda).
\end{align}

\subsection{Some known results}
 Generally, for a restricted Lie algebra $\mathscr{L}$, we call $\mathscr{L}$ a unipotent Lie algebra if, for all $x \in \mathscr{L}$, there exists $r>0$ such that $x^{[p^r]}=0$ where $x^{[p^r]}$ denotes the $p$-th power map iterated $r$ times. So clearly, $\mathscr{L}$ is unipotent also means $\mathscr{L}$ is $p$-nilpotent. We define $\mathscr{L}_\chi:=\{X \in \mathscr{L} \mid \chi([X,\mathscr{L}])=0\}$. The following basic result will be used latter.

\begin{lemma} (\cite[Corollary 7.2]{Jan1}) \label{lem: 1.2.1} Let $\frak m$ be a uipotent restricted subalgebra of $\ggg$ with $\frak m \cap \ggg_\chi=\{0\}$. If $\chi([\frak m,\frak m])=0$ and $\chi(\frak m^{[p]})=0$, then $U_\chi(\ggg)$-module is free over $U_\chi(\frak m)$.
\end{lemma}

 \section{Irreducible modules of $\frak {sl}(n)$ for regular nilpotent case with $p \mid n$}
\subsection{Restrictions and inductions}
Let $z=E_{11}$, then $\gl_{n}=\sl_n \oplus \mathbb{K}z$ and $\gl_n \slash \sl_n \cong \mathbb{K}z$. For $\chi \in \gl_{n}^*$, we still denote $\chi|_{\sl_n}$ by $\chi$, thus $U_{\chi|_{\sl_n}}(\frak {sl}_n)=U_{\chi}(\frak {sl}_n)$.
Each $\frak {gl}_n$-module $N$ can be regarded as a $\frak {sl}_n$-module, we use $\res N$ to denote this $\frak {sl}_n$-module. Similarly, each $U_{\chi}(\frak {gl}_n)$-module $M$ can be regarded as a $U_{\chi}(\frak {sl}_n)$-module, we use $\res_{\chi}M$ to denote this $U_{\chi}(\frak {sl}_n)$-module, where $\chi \in \gl_{n}^*$.
Thus we can define two functors:
$$\res_{\chi}:  \{U_{\chi}(\frak {gl}_n)-modules\} \rightarrow \{U_{\chi}(\frak {sl}_n)-modules\}$$
$$\ M  \mapsto  \res_{\chi}M$$
$$\ind_{\chi}:  \{U_{\chi}(\frak {sl}_n)-modules\} \rightarrow \{U_{\chi}(\frak {gl}_n)-modules\}$$
$$\ M  \mapsto  U_\chi(\frak {gl}_n)\otimes_{U_\chi(\frak {sl}_n)}M$$

For each $U_{\chi}(\frak {gl}_n)$-module $M$ and each $U_{\chi}(\frak {sl}_n)$-module $N$, we have the reciprocity:
 \begin{align}\label{eq: 2.1}
 {\Hom}_{U_{\chi}(\frak {gl}_n)}(\ind_{\chi}M, N) \cong {\Hom}_{U_{\chi}(\frak {sl}_n)}(M, \res_{\chi}N).
 \end{align}

Namely $\ind_{\chi}$ is left adjoint to $\res_{\chi}$.

\subsection{Restriction functors}
Suppose $M$ and $N$ are $\frak {gl}_n$-modules, then ${\Hom}_k(M, N)$ becomes a $\frak {gl}_n$-module with $\frak {gl}_n$ action: $(x.f)(m)=x.f(m)-f(x.m)$ for each $x\in \frak {gl}_n$, $f\in {\Hom}_k(M, N)$ and $m\in M$.
 We have the following observations:

(a) For given $\gl_n$-modules $M$, $N$, ${\Hom}_{\frak {sl}_n}(\res M, \res N)$ becomes the submodule of ${\Hom}_k(M, N)$(as $\frak {gl}_n$-modules). Note that the action of $\frak {sl}_n$ on ${\Hom}_{\frak {sl}_n}({\res}M, {\res}N)$ is trival. So it is a ${\frak {gl}_n}/{\frak {sl}_n}$-module.

(b) Furthermore, suppose $\chi\in\gl_n^*$,  and $M$, $N$ are $U_{\chi}(\frak {gl}_n)$-modules, then ${\Hom}_{\frak {sl}_n}({\res}_{\chi}M, {\res}_{\chi}N)$ becomes a $U_0({\frak {gl}_n}/\frak {sl}_n)$-module.

In following, we suppose that $M$, $N$ are finite dimensional $U_{\chi}({\frak {gl}_n})$-modules.
Consider $z$ acts on ${\Hom}_{\frak {sl}_n}({\res}_{\chi}M, {\res}_{\chi}N)$, suppose
$f\in \Hom_{\frak {sl}_n}({\res}_{\chi}M, {\res}_{\chi}N)$ in the eigenspace for the eigenvalue $a$. Since ${\Hom}_{{\sl_n}}({\res}_{\chi}M, {\res}_{\chi}N)$ is $U_0({\frak {gl}_n}/\frak {sl}_n)$-module and $z^{[p]}=z$, then $0=(z^p-z^{[p]}).f=(z^p-z).f=(a^p-a)f$. This implies that $a^p=a$, so $a\in \mathbb{F}_p$. We define
$${\Hom}_{\frak {sl}_n}({\res}_{\chi}M, {\res}_{\chi}N)_a:=\{f \in\Hom_{\frak {sl}_n}({\res}_{\chi}M, {\res}_{\chi}N)\mid
z.f=af \}$$
 for all $a\in \mathbb{F}_p$. Since $z$ acts diagonally with eigenvalues in $\mathbb{F}_p$, then ${\Hom}_{\frak {sl}_n}({\res}_{\chi}M, {\res}_{\chi}N)$ is the direct sum of its eigenspces with eigenvalues in $\mathbb{F}_p$. Hence we get

\begin{align} \label{sum}
\Hom_{\frak {sl}_n}({\res}_{\chi}M, {\res}_{\chi}N)=\bigoplus_{a\in \mathbb{F}_p}\Hom_{\frak {sl}_n}({\res}_{\chi}M, {\res}_{\chi}N)_a
\end{align}

For $a \in \bbk$, $\bbk_a$ is a one dimensional $\frak {gl}_n$-module where $\sl_n$ acts 0 and where $z$ acts via $a$.

\begin{lemma}\label{lem: 2.2.2} For ${\chi}, {\chi}' \in \frak {gl}_n^*$, if $M$ is a $U_{\chi}(\frak {gl}_n)$-module, $N$ is a $U_{{\chi}'}(\frak {gl}_n)$-module, then $M \otimes N$ is $U_{{\chi}+{\chi}'}(\frak {gl}_n)$-module.
\end{lemma}

For a given $a\in \mathbb{F}_p$, $U_0(\frak {gl}_n)$ admits a one-dimensional module $\bbk_a$ with trival $\sl_n$-action and $z$-action by scalar $a$. Then $M\otimes \bbk_a$ is an $U_{\chi}(\frak {gl}_n)$-module. Moreover if $M$ is a simple $U_{\chi}(\frak {gl}_n)$-module, then $M\otimes \bbk_a$ is also a simple $U_{\chi}(\frak {gl}_n)$-module with $a\in \mathbb{F}_p$. Similarly, $N \otimes \bbk_{-a}$ is an $U_{\chi}(\frak {gl}_n)$-module with $a\in \mathbb{F}_p$.


\begin{lemma}\label{lem: 2.2.1} Suppose $M$, $N$ are $U_{\chi}({\frak {gl}_n})$-modules, ${\chi}\in \frak {gl}_n^*$. For $a\in \mathbb{F}_p$, we get isomorphisms (as vector spaces)
\begin{align*}
 {\Hom}_{\frak {sl}_n}({\res}_{\chi}M, {\res}_{\chi}N)_a \cong \Hom_{\gl_n}(M \otimes \bbk_a, N) \cong \Hom_{\gl_n}(M, N \otimes \bbk_{-a}).
\end{align*}
\end{lemma}

By Lemma \ref{lem: 2.2.1}, we have that if $M$, $N$ are finite dimensional $U_{\chi}(\frak {gl}_n)$-modules with ${\chi}\in \frak {gl}_n^*$, then
\begin{align}
\Hom_{\sl_n}({\res}_{\chi}M, {\res}_{\chi}N) \cong \bigoplus_{a\in \mathbb{F}_p}\Hom_{\frak {gl}_n}(M \otimes \bbk_a, N) \cong \bigoplus_{a\in \mathbb{F}_p}\Hom_{\frak {gl}_n}(M, N\otimes \bbk_{-a})
\end{align}

as linear spaces.



\begin{lemma}\label{lem: 2.2.3} Suppose $M$ is a $U_{\chi}({\frak {gl}_n})$-module. Then

\begin{itemize}
\item[(1)] For $a, b \in \mathbb{F}_p$, $\bbk_a \otimes \bbk_b \cong \bbk_{a+b}$.

\item[(2)] The following isomorphisms are equivalent:

(i)$M \otimes \bbk_1 \cong M$;

(ii)$M \otimes \bbk_a \cong M$, $\forall {a\in \mathbb{F}_p}$;

(iii)$M \otimes \bbk_a \cong M $, for some ${a\in \mathbb{F}_p^\times}:=\mathbb{F}_p \backslash \{0\}$.
    \end{itemize}
\end{lemma}

 All the isomorphisms in Lemma \ref{lem: 2.2.3} are the isomorphisms of $U_{\chi}({\frak {gl}_n})$-modules.

\begin{proposition}\label{Prop: 2.2.4} Suppose that ${\chi}\in \frak {gl}_n^*$, $E$ is a simple $U_{\chi}(\frak {gl}_n)$-module, then
\begin{equation*}
	\dim \Hom_{\frak {sl}_n}({\res}_{\chi}E, {\res}_{\chi}E)=\begin{cases}
		1, & \text{if $E \ncong E \otimes \bbk_1$,}\\
		p, & \text{if $E \cong E \otimes \bbk_1$.}
	\end{cases}
\end{equation*}
\end{proposition}

\begin{proof}
	$\Hom_{\frak {sl}_n}({\res}_{\chi}E, {\res}_{\chi}E) \cong \bigoplus_{a\in \mathbb{F}_p}\Hom_{\frak {gl}_n}(E \otimes \bbk_a, E)$
	\begin{equation*}
		=\begin{cases}
			\Hom_{{\frak {gl}_n}}(E \otimes \bbk_a, E), & \text{if $E \ncong E \otimes \bbk_1$,}\\
			 \Hom_{\frak {gl}_n}(E \otimes \bbk_a, E)^{\bigoplus p}, & \text{if $E \cong E \otimes \bbk_1$.}
		\end{cases}
	\end{equation*}
	Since $E$ is a simple $U_{\chi}(\frak {gl}_n)$-module, $\dim\Hom_{\frak {gl}_n}(E, E)=1$.
\end{proof}

\subsection{A criterion}
Let $\varpi_n=\varepsilon_1+\cdots+\varepsilon_n$ and $W$ be the Weyl group of $\gl_n$, then for any $\tau \in W$, we have $\tau(\varpi_n)=\varpi_n$. We assume that $\chi \in \frak {gl}_n^*$ is of standard Levi form with $I \subseteq \{\alpha_1,\ldots,\alpha_{n-1}\}$. We have the triangular decomposition $\frak {gl}_n=\frak n^-\oplus\frak h\oplus\frak n^+=\frak n^-\oplus\frak b$, and set $\Lambda=\sum_{i=1}^{n}\mathbb{F}_p\varepsilon_{i} \subsetneq \hhh^*$ as in \S\ref{SAC}.

\begin{lemma}\label{lem: 2.3.2} Let  $\lambda \in \Lambda$, $a \in\mathbb F_p$, then $L_{\chi}(\lambda) \otimes \bbk_a \cong L_{\chi}(\lambda+a \varpi_n)$.
\end{lemma}

\begin{proof}
	 Let $v_\lambda$ be the maximal vector of $Z_{\chi}(\lambda)$. Then for $h \in \frak h$, $h.(v_\lambda \otimes v_a)=h.v_\lambda \otimes v_a+v_\lambda \otimes h.v_a=\lambda(h)v_\lambda \otimes v_a+v_\lambda \otimes a \varpi_n(h)v_a=(\lambda+a \varpi_n)(h)v_\lambda \otimes v_a$. So it is easy to conclude that $Z_{\chi}(\lambda) \otimes \bbk_a \cong Z_{\chi}(\lambda+a \varpi_n)$.
	
	We just need to show that $L_{\chi}(\lambda) \otimes \bbk_a$ is simple. If $L_{\chi}(\lambda) \otimes \bbk_a$ has a submodule $M$, then $L_{\chi}(\lambda)$ has $M \otimes \bbk_{-a}$ as its submodule, this yields that $L_{\chi}(\lambda) \otimes \bbk_a$ is irreducible.
\end{proof}

From now on, we always assume that $\chi \in \frak {gl}_n^*$ is regular nilpotent in this section. The restriction $\chi|_{\sl_n}$ is still regular nilpotent. Actually, the meaning of $\chi$ being regular nilpotent is equivalent to the meaning of $\chi|_{\sl_n}$ being regular nilpotent.  Then ${\chi}$ has standard Levi form with $\{\alpha_1,\ldots,\alpha_{n-1}\}$, namely ${\chi}(X_{-\alpha_i})=1$ for $1 \leq i \leq {n-1}$ while ${\chi}(X_\alpha)=0$ for all other ${\alpha}\in R$ and ${\chi}(\frak h)=0$.

\begin{proposition}\label{prop: 2.3.3} For $\lambda \in \Lambda$, there is an isomophism $L_\chi(\lambda) \otimes \bbk_1 \cong L_\chi(\lambda)$ of $U_\chi(\frak {gl}_n)$-modules if and only if $\lambda \in W.0$.
\end{proposition}

\begin{proof}
	Suppose $\lambda \in W.0$, this is to say, there exists $w \in W$ such that $\lambda=w.0$, i.e. $\lambda+\rho=w(\rho)$. Keep in mind $\rho=\sum_{j=1}^{n}(n-j)\varepsilon_j$. Because $p \mid n$, we have $n=0$ in $\mathbb{K}$. It follows that $\rho+\varpi_n=\sum_{j=1}^{n}(n+1-j)\varepsilon_j=n\varepsilon_1+\sum_{j=1}^{n-1}(n-j)\varepsilon_{j+1}=\sum_{j=1}^{n-1}(n-j)\varepsilon_{j+1}$. So there exists $\sigma=(12 \cdots n) \in W$ such that $\rho+\varpi_n=\sigma(\rho)$, i.e. $\varpi_n=\sigma.0$. Thus $\lambda+\varpi_n=\lambda+w(\varpi_n)=\lambda+w(\sigma(\rho)-\rho)=\lambda+w\sigma(\rho)-w(\rho)=w\sigma(\rho)-\rho=w\sigma w^{-1}(\lambda+\rho)-\rho$, namely $\lambda+\varpi_n=(w\sigma w^{-1}).\lambda$. So Lemma \ref{lem: fundamental iso 2} and Lemma \ref{lem: 2.3.2} yield that $L_\chi(\lambda) \otimes \bbk_1 \cong L_\chi(\lambda)$.
	
	Conversely, if $L_\chi(\lambda) \otimes \bbk_1 \cong L_\chi(\lambda)$, then there exists $\sigma \in W$ such that $\lambda=\sigma.(\lambda+\varpi_n)$ (By Lemma \ref{lem: fundamental iso 2} and Lemma \ref{lem: 2.3.2} again), namely $\lambda+\rho=\sigma(\lambda+\varpi_n+\rho)$. Assume that $\lambda+\rho=\sum_{i=1}^{n}a_i\varepsilon_i$ with $a_i \in \mathbb{F}_p$ ($1 \leq i \leq n$), then $\lambda+\rho+\varpi_n=\sum_{i=1}^{n}(a_i+1)\varepsilon_i$ and $\sigma^{-1}(\lambda+\rho)=\sum_{i=1}^{n}a_i\varepsilon_{\sigma^{-1}(i)}=\sum_{i=1}^{n}a_{\sigma(i)}\varepsilon_i$. Thus $\sigma^{-1}(\lambda+\rho)=\lambda+\varpi_n+\rho$ implies that 	
	\begin{equation}\label{eq: index}
	a_{\sigma(i)}=a_i+1 \text{ for } 1\leq i \leq n.
	\end{equation}
	Let $\sigma=(i_1^1 \cdots i_{s_1}^1)(i_1^2\cdots i_{s_2}^2)\cdots(i_1^m \cdots i_{s_m}^m)$ with $s_1+\cdots+s_m=n$, then $\sigma^{s_1}(i_1^1)=i_1^1$, we have $a_{\sigma ^{s_1}(i_1^1)}=a_{i_1^1}$. By (\ref{eq: index}), it follows that $a_{i_1^1}=a_{\sigma ^{s_1}(i_1^1)}=a_{\sigma ^{s_1-1}(i_1^1)}+1=\ldots=a_{i_1^1}+s_1$, so $p$ divides $s_1$. Similarly $p$ divides $s_j$ for $1\leq j\leq m$. Thanks to (\ref{eq: index}), there exists $b_1,\ldots,b_m \in \mathbb{F}_p$, such that $\mu=(b_1,b_1+1,\ldots,b_1+s_1-1,b_2,b_2+1,\ldots,b_2+s_2-1,\ldots,b_m,b_m+1,
\ldots,b_m+s_m-1) \in W(\lambda+\rho)$. Since $p$ divides $s_j$ with $1\leq j\leq m$, each element in $\mathbb{F}_p$ occurs $n/p$ times in $b_1,b_1+1,\ldots,b_1+s_1-1,b_2,b_2+1,\ldots,b_2+s_2-1,\ldots,b_m,b_m+1,\ldots,
b_m+s_m-1$. We finally conclude that each element in $\mathbb{F}_p$ occurs $n/p$ times from 0 to $n-1$. Correspondingly, there exists $\tau \in W$ such that $\mu=\tau(\rho)$, i.e. $\mu \in  W(\rho)$. So $\mu \in  W(\rho)$ and $\mu \in  W(\lambda+\rho)$ implies that $\lambda+\rho \in W(\rho)$. This means that there exists $w \in W$, such that $\lambda+\rho=w(\rho)$. This yields that $\lambda=w.0$, namely $\lambda \in W.0$.	
\end{proof}

\subsection{Some properties on the functors $\res$ and $\ind$}

In this section we always assume that $p \mid n$ and $\chi \in \frak {sl}_n^*$ is regular nilpotent.

\begin{proposition}\label{prop: 2.4.1} Let $M$ be an $U_{\chi}(\frak {sl}_n)$-module, then the composition factors of ${\res}_{\chi}$${\ind}_{\chi}M$ are all isomorphic to $M$.
\end{proposition}

\begin{proof}
	Let $M_i=\sum_{0 \leq j \leq i}h^j \otimes M \subseteq {\ind}_\chi M$ where $h=E_{1,1}$, we can see that $M_i$ is $U_\chi(\frak {sl}_n)$-module, not $U_\chi(\frak {gl}_n)$-module. Hence we get a $U_\chi(\frak {sl}_n)$ as follows:
	$$0=M_0 \subset M_1 \subset \cdots \subset M_p={\res}_{\chi}{\ind}_{\chi}M$$ with $M_i/M_{i-1} \cong M$.
\end{proof}

\begin{corollary}\label{cor: 2.4.2} Let $M$ be a $U_{\chi}(\frak {sl}_n)$-module, and denote by $[M]$ the corresponding element in the Grothendieck group of the $U_\chi(\sl_n)$-module category. Then
$[{\res}_{\chi}{\ind}_{\chi}M]=p[M]$.
\end{corollary}

\begin{proof}
	Since ${\dim}({\ind}_{\chi}M)=p\dim M$, it is obviously by Proposition \ref{prop: 2.4.1}.
\end{proof}

\begin{corollary}\label{cor: 2.4.3} Suppose that $M$ is a simple $U_{\chi}(\frak {gl}_n)$-module. Then there exist a simple $U_{\chi}(\frak {sl}_n)$-module $L$ and a positive integer $m$, such that $[{\res}_{\chi}M]=m[L]$.
\end{corollary}

\begin{proof}
	Let $L$ be a irreducible submodule of ${\res}_{\chi}M$, then $\Hom_{U_{\chi}(\frak {sl}_n)}(L, {\res}_{\chi}M)\neq 0$. Using the reciprocity in \eqref{eq: 2.1}, we have
$$\Hom_{U_{\chi}(\frak {gl}_n)}({\ind}_{\chi}L, M)=\Hom_{U_{\chi}(\frak {sl}_n)}(L, {\res}_{\chi}M)\neq 0.$$
 so $M$ is a composition factor of ${\ind}_{\chi}L$. By Proposition \ref{prop: 2.4.1}, each composition factor of ${\res}_{\chi}{\ind}_{\chi}L$ is isomorphic to $L$. Thus, there exist an ${U_{\chi}(\frak {sl}_n)}$-module $E$ and a positive integer $m'$, such that $m'[L]=[{\res}_{\chi}{\ind}_{\chi}L]=[{\res}_{\chi}M]+[E]$. Thus there exists a positive integer $m \textless m'$, such that $[{\res}_{\chi}M]=m[L]$.
\end{proof}

\begin{remark} The proof of Corollary \ref{cor: 2.4.3} also implies that $L$ is a composition factor of ${\res}_{\chi}M$ if and only if $M$ is a composition factor of ${\ind}_{\chi}L$.
\end{remark}


\begin{lemma}\label{lem: 2.4.4} Let $A$ be a finite dimensional associative algebra over $\bbk$, $M$ be a finite dimensional $R$-module. Then $M$ is indecomposable if and only if $\text{End}_A(M)$ is local ring.
\end{lemma}

\begin{proof}
	Recall that $\text{End}_A(M)$ is a local ring if and only if for all $x \in A$, either $x$ or $1-x$ is unit. Then Fitting Lemma yields the following conclusion.
\end{proof}

\begin{lemma} \rm{(\cite[Theorem 1.3]{Rol})} \label{lem: 2.4.5}  Let $K$ be a $p$-group, $e$ be the identity element  of $K$. Assume that $A=\bigoplus_{x \in K}A_x$ is a $K$-graded algebra over $\bbk$. If $A_e$ is a local ring, so is $A$.
\end{lemma}

\begin{proposition}\label{prop: 2.4.6} Let ${\chi}\in \frak {gl}_n^*$, $M$ be an indecomposable $U_{\chi}(\frak {gl}_n)$-module. Then ${\res}_{\chi}M$ is also an indecomposable $U_{\chi}(\frak {sl}_n)$-module.
\end{proposition}

\begin{proof} Note that
	$\text{End}_{\frak {sl}_n}(M)=\bigoplus_{a\in \mathbb{F}_p}\text{End}_{\frak {sl}_n}(M)_a=\bigoplus_{a\in \mathbb{F}_p}\text{End}_{\frak {gl}_n}(M \otimes \bbk_a, M)$. So $\text{End}_{\frak {sl}_n}(M)$ is $\mathbb{F}_p$-graded. As $M$ is an indecomposable $U_{\chi}(\frak {gl}_n)$-module, $\text{End}_{\frak {sl}_n}(M)_0=\text{End}_{\frak {gl}_n}(M)$ is local ring(by Lemma \ref{lem: 2.4.4}). By appling Lemma \ref{lem: 2.4.5}, we have $\text{End}_{\frak {gl}_n}(M)$ is a local ring. Now Lemma \ref{lem: 2.4.4} yields the conclusion.
\end{proof}

\begin{proposition}\label{prop: 2.4.7} If $P$ is an indecomposable projective module over $U_{\chi}(\frak {gl}_n)$, then ${\res}_{\chi}P$ is also an indecomposable projective module over $U_{\chi}(\frak {sl}_n)$.
\end{proposition}

\begin{proof} Note that $U_{\chi}(\frak {gl}_n)$ is free over $U_{\chi}(\frak {sl}_n)$. So
	${\res}_{\chi}P$ is a projective $U_{\chi}(\frak {sl}_n)$-module. Proposition \ref{prop: 2.4.6} yields that ${\res}_{\chi}P$ is indecomposable.
\end{proof}

As a corollary, we have
\begin{corollary}\label{cor: 2.4.8} Let $M$ be a simple $U_{\chi}(\frak {gl}_n)$-module, $P_M$ be its projective cover. Then there exist a simple $U_{\chi}(\frak {sl}_n)$-module $E$, such that ${\res}_{\chi}P_M$ is its projective cover.
\end{corollary}

\begin{lemma}\label{lem: 2.4.9}(\cite[Lemma 1.7.6]{Ben}) Let $A$ be a finite dimensional associative algebra over $\bbk$, $E$ be a simple module over $A$. Suppose that $M$ is an $A$-module with finite length which is denoted by $l(M)$, $P_E$ is the projective cover of $E$. Then $$[M:E]=\dim\Hom_A(P_E, M).$$
\end{lemma}

Keep in mind that we have triangular decomposition $\frak {gl}_n=\frak n \oplus \frak h \oplus \frak {n^-}, \frak {sl}_n=\frak n \oplus \frak h' \oplus \frak {n^-}$ where $\hhh'$ is a codimension-one subspace (subalgebra) of $\hhh$ with $\hhh=\hhh' \oplus \mathbb{K}z$. Let $W$ be the Weyl group of $\sl_n$.

\begin{lemma}\label{lem: 2.4.10}
Let $\chi \in \frak {gl}^*_n, \lambda \in \frak h^*$. The following statements hold.
\begin{itemize}
\item[(1)] $\chi \in \frak {gl}^*_n$ is regular(resp. subregular) nilpotent if and only if $\chi|_{\frak {sl}_n} \in \frak {sl}_n^*$ is regular(resp. subregular) nilpotent;
\item[(2)] $\lambda \in W.0 \subset \frak h^*$ if and only if $\lambda |_{\hhh'} \in W.0 \subset \frak h'^*$;
\item[(3)] If $\chi$ is nilpotent and $\lambda |_{\frak h'} \in \mathscr X_p:=\sum_{i=1}^{n-1}\mathbb{F}_p\alpha_i|_{\frak {h}'}$, then ${\res}_{\chi}Z_\chi(\lambda)=Z_{\chi |_{\frak {sl}_n}}(\lambda |_{\frak h'})$.
    \end{itemize}
\end{lemma}

\begin{proof}
		Let $w \in W$, $\lambda \in \hhh^*$ and $\mu \in \hhh^*$. Then $\lambda=w.\mu$ if and only if $\lambda|_{\frak h'}=w.\mu|_{\frak h'}$. This yields (2).
\end{proof}

\begin{lemma}\label{lem: 2.4.11}
If $\chi \in \frak {gl}_n^*$ is regular nilpotent, the baby Verma module $Z_{\chi}(\lambda)$ is simple $U_\chi(\frak {gl}_n)$-module.
\end{lemma}

\begin{proof}
	If $p$ does not divide $n$, then work of Friedlander and Parshall (in \cite{FP1} and \cite{FP2}) shows that both $Z_\chi(\lambda)$ and its restriction to $\frak {sl}_n$ are simple for each $\lambda \in \Lambda$; If $p$ divides $n$, the each $Z_\chi(\lambda)$ is still simple, cf. \cite[Proposition 10.5]{Jan1}.
\end{proof}

Now we are in the position to prove  Theorem \ref{thm: 2.4.12}.

\subsection{The proof of Theorem \ref{thm: 2.4.12}} \label{sec: proof of thm 0.2}
	We can regard $\chi$ and $\lambda$ as the elements in $\frak {gl}_n^*$ and $\Lambda=\sum_{i=1}^{n}\mathbb{F}_p\varepsilon_i \subsetneq \hhh^*$ respectively, and denoted them by $\underline{\chi}$ and $\underline{\lambda}$ respectively. Lemma \ref{lem: 2.4.11} implies that $E=Z_{\underline{\chi}}(\underline{\lambda})$ is a simple $U_{\chi}(\frak {gl}_n)$-module, namely $Z_{\underline{\chi}}(\underline{\lambda})=L_{\underline{\chi}}(\underline{\lambda})$. Let $P_E$ be the projective cover of $E$ as $U_{\chi}(\frak {gl}_n)$-module. Corollary \ref{cor: 2.4.8} yeilds that there exists a simple $U_{\chi}(\frak {sl}_n)$-module $L$, such that ${\res}_{\underline{\chi}}P_E$ is the projective cover of $L$. Let $a \in \mathbb{F}_p$, if $E$ is not isomorphic to $E \otimes \bbk_a$, then $\Hom_{\ggg}(P_E, E \otimes \bbk_a)=0$; if $E$ is isomorphic to $E \otimes \bbk_a$, then ${\dim\Hom}_{\frak {gl}_n}(P_E, E \otimes \bbk_a)={\dim\Hom}_{\frak {gl}_n}(P_E, E)=1$. Thus we have $$[{\res}_{\underline{\chi}}E: L]\xlongequal[]{\text{Lemma }\ref{lem: 2.4.9}}{\dim\Hom}_{\frak {sl}_n}({\res}_{\underline{\chi}}P_E, {\res}_{\underline{\chi}}E)\xlongequal[]{(\ref{sum})}\bigoplus_{a\in \mathbb{F}_p}{\dim\Hom}_{{\frak {sl}_n}}({\res}_{\underline{\chi}}P_E, {\res}_{\underline{\chi}}E)_a$$
	\begin{equation*}\xlongequal[]{\text{Lemma }\ref{lem: 2.2.1}}\bigoplus_{a\in \mathbb{F}_p}{\dim\Hom}_{{\frak {gl}_n}}(P_E, E \otimes \bbk_a)=\begin{cases}
			1, & \text{if $E \otimes \bbk_1 \ncong E$,}\\
			p, & \text{if $E \otimes \bbk_1 \cong E$}.
		\end{cases}
	\end{equation*}
	
Proposition \ref{prop: 2.3.3} implies that $\underline{{\lambda}} \in W.0$ if and only if $E \otimes \bbk_1 \cong E$; $\underline{{\lambda}} \notin W.0$ if and only if $E \otimes \bbk_1 \ncong E$. Lemma \ref{lem: 2.4.10}(3) implies that $\text{res}_{\underline{\chi}}E=\text{res}_{\underline{\chi}}Z_{\underline{\chi}}(\underline{\lambda})=Z_{\underline{\chi}|_{\frak {sl}_n}}(\underline{\lambda}|_{\hhh'})=Z_\chi(\lambda)$.

If ${\lambda} \in W.0$, we have $\underline{{\lambda}} \in W.0$ (By Lemma \ref{lem: 2.4.10}), then $E \otimes \bbk_1 \cong E$, thus $[\text{res}_{\underline{\chi}}E: L]=p$. By Corollary \ref{cor: 2.4.3}, we have $[\text{res}_{\underline{\chi}}E]=p[L]$. Thus $\text{res}_{\underline{\chi}}E=Z_\chi(\lambda)$ implies that $[Z_\chi(\lambda)]=p[L]$.

If ${\lambda} \notin W.0$, we have $\underline{{\lambda}} \notin W.0$ (See Lemma \ref{lem: 2.4.10}), then $E \otimes \bbk_1 \ncong E$, thus $[\text{res}_{\underline{\chi}}E: L]=1$. By Corollary \ref{cor: 2.4.3}, we have $[\text{res}_{\underline{\chi}}E]=[L]$, namely $\text{res}_{\underline{\chi}}E \cong L$. Thus $\text{res}_{\underline{\chi}}E=Z_\chi(\lambda)$ implies that $Z_\chi(\lambda) \cong L$.	

Since $L$ is the composition factor of $Z_\chi(\lambda)$, then $L$ isomorphic to the unique simple quotient of $Z_\chi(\lambda)$, i.e. $L_\chi(\lambda) \cong L$ as $U_{\chi}(\frak {sl}_n)$-modules.

\begin{remark}\label{remark of Theorem 0.3}
	By Theorem \ref{thm: 2.4.12}, we conclude the following result:
	
	Assume that $\chi \in \frak{sl}_n^*$ is regular nilpotent with $p \mid n$ and $\lambda \in \mathscr{X}_p=\sum_{i=1}^{n-1}\mathbb{F}_p\alpha_i$. Let $N=\abs{R^+}$. If $\lambda \in W.0$, then $\text{dim}L_\chi(\lambda)=p^{N-1}$; if $\lambda \notin W.0$, $\text{dim}L_\chi(\lambda)=p^N$.
\end{remark}


\subsection{Descriptions of simple modules in the regular nilpotent case}\label{sec: 2.5}
In this subsection we always assume that $p \mid n$ and $\chi \in \frak {sl}_n^*$ is regular nilpotent of the form \eqref{eq: regular} if not particularly indicated. Thus $W_I=W$ (here $W$ is the Weyl group of $\sl_n$). Now we give a description of irreducible module $L_\chi(\lambda)$. When $\lambda \notin W.0$, by Theorem \ref{thm: 2.4.12}, $L_\chi(\lambda)$ coincides with $Z_\chi(\lambda)$. We do not do anything in this case. So we only need to describe $L_\chi(\lambda)$ when $\lambda \in W.0$. For a given $\lambda \in W.0$, by Lemma \ref{lem: fundamental iso}, we have $Z_\chi(\lambda) \cong Z_\chi(0)$ and $L_\chi(\lambda) \cong L_\chi(0)$ (Note $W_I=W$ in this case). Therefore, in order to understand the structure of $L_\chi(\lambda)$ for $\lambda\in W.0$,  we only need consider $\lambda=0$ in the sequel.

Let $N$ be the number of positive roots of $\ggg=\sl_n$. Obviously, we have  $N=\frac{1}{2}n(n-1)$.

\subsubsection{}
Let us introduce the following observation.

\begin{lemma}\label{cor: 2.5.1} Set $\ggg=\frak {sl}_n(\bbk)$ and $\chi \in \frak {sl}_n^*$ is regular nilpotent of the form \eqref{eq: regular}. If $M$ is a non-zero $U_\chi(\ggg)$-module, then $p^{N-1} \mid \dim M$.
\end{lemma}

\begin{proof}
	Let $\frak m=\text{span}\{X_{-\alpha} \mid \alpha \in R^+, \alpha \neq\alpha_{1}\}=\bigoplus \limits_{\alpha \in R^+, \alpha \neq\alpha_{1}}\bbk X_{-\alpha}=\bigoplus \limits_{\alpha \in R^+, \alpha \neq\alpha_{1}}\ggg_{-\alpha}$. It is clear that $\frak m$ is unipotent (i.e. $\frak m$ is $p$-nilpotent). 	
	Suppose $x \in \frak m \cap \ggg_\chi$, then $x=\sum\limits_{\alpha \in R^+, \alpha \neq\alpha_{1}}k_{-\alpha}X_{-\alpha}$ where $k_{-\alpha} \in \bbk$, and $\chi([x,\ggg])=0$. Note that $\chi([x,X_{\alpha_{1}}])=0, \chi([x,X_{\alpha_{2}}])=0,\cdots,\chi([x,X_{\alpha_{n-2}}])=0$. So we have $k_{-\alpha_{1}-\alpha_{2}}=0,   k_{-\alpha_{2}-\alpha_{3}}=0,\cdots,k_{-\alpha_{n-2}-\alpha_{n-1}}=0$.
	
In view that $\chi([x,X_{\alpha_{1}+\alpha_{2}}])=0,\ldots,\chi([x,X_{\alpha_{n-3}+\alpha_{n-2}}])=0$, we have
\begin{align}
&k_{-\alpha_{1}-\alpha_{2}-\alpha_{3}}=0,\cdots,\cr
&k_{-\alpha_{n-3}-\alpha_{n-2}-\alpha_{n-1}}=0,\cdots .
	\end{align}
	Furthermore, $\chi([x,X_{\alpha_{1}+\cdots+\alpha_{n-2}}])=0$, so $k_{-\alpha_{1}-\cdots-\alpha_{n-1}}=0$. 	So $x=\sum\limits_{\alpha \neq \alpha_{1}, \alpha \in S} k_{-\alpha}X_{-\alpha}$ where $S=\{\alpha_{1},\cdots,\alpha_{n-1}\}$. Continue to take use of the condition $\chi([x,h_{\alpha_{1}}])=0,\cdots,\chi([x,h_{\alpha_{n-2}}])=0$. We have that $k_{-\alpha_{2}}=0,\cdots,k_{-\alpha_{n-1}}=0$. 	Hence we finally have $x=0$. This implies $\frak m \cap \ggg_\chi=\{0\}$.
	Note that $\dim U_\chi(\frak m)=p^{N-1}$ and $M$ is free over $U_\chi(\frak m)$ (Lemma \ref{lem: 1.2.1}). We finally have $p^{N-1} \mid \dim M$.
\end{proof}

As a consequence of Lemma \ref{cor: 2.5.1}, we have
\begin{lemma}\label{lem: 2.5.3} Set $\ggg=\frak {sl}_n(\bbk)$. Suppose $\chi \in {\ggg}^*$ has the form \eqref{eq: regular} and $M$ is a non-zero $U_\chi(\ggg)$-module. If $\dim M=p^{N-1}$, then $M$ is irreducible.
\end{lemma}

\begin{lemma}\label{2.25}
	Set $\ggg=\frak {sl}_n(\bbk)$ with $p \nmid n$ and $\chi \in \frak {sl}_n^*$ is subregular nilpotent of the form \eqref{eq: subregular}. If $M$ is a $U_\chi(\ggg)$-module, then $p^{N-1} \mid \dim M$.
\end{lemma}

\begin{proof}
	We use \cite[Theorem 2.6(a)]{Jan3} to conclude that $p^{N-1}$ divides the dimension of a simple $U_\chi(\ggg)$-module. Hence  $p^{N-1} \mid \dim M$ (Using the composition factors of $M$).
\end{proof}

Thus, we have

\begin{lemma}\label{2.26}
	Set $\ggg=\frak {sl}_n(\bbk)$. Suppose $\chi \in \frak {sl}_n^*$ is subregular nilpotent with $p \nmid n$ and $M$ is a non-zero $U_\chi(\ggg)$-module. If $\dim M=p^{N-1}$, then $M$ is irreducible.
\end{lemma}


For the time being, we let
$$\ggg=\frak {sl}_n \mbox{ with simple roots }\alpha_{1},\ldots,\alpha_{n-1}$$ and
$$\widetilde \ggg=\frak {sl}_{n+1}\mbox{ with simple roots }\alpha_{1},\ldots,\alpha_{n}.$$

We have an injective homomorphism (of Lie algebra) $\iota: \ggg \rightarrow \widetilde{\ggg}$ with $A \mapsto \begin{bmatrix} A &  \\  & 0 \end{bmatrix}$, thus we can consider $\ggg$ as a subalgebra of $\widetilde{\ggg}$, namely $\ggg \subsetneq \widetilde{\ggg}$.

Keep in mind that $\ggg=\nnn^+ \oplus \hhh \oplus \nnn^-$ with $\hhh=\bigoplus_{i=1}^{n-1}\bbk(E_{ii}-E_{i+1,i+1})$, and $\widetilde{\ggg}=\widetilde{\nnn}^+ \oplus \widetilde{\hhh} \oplus \widetilde{\nnn}^-$ with $\widetilde{\hhh}=\hhh \bigoplus \bbk(E_{nn}-E_{n+1,n+1})$. $\hhh$ (resp.$\widetilde{\hhh}$) is the cartan subalgebra of $\ggg$ (resp. $\widetilde{\ggg}$) which consists of all the diagonal matrices of $\ggg$ (resp. $\widetilde{\ggg}$). Let $\bbb^+:=\nnn^+ \oplus \hhh$ and $\widetilde{\bbb}^+:=\widetilde{\nnn}^+ \oplus \widetilde{\hhh}$ that are Borel subalgebras of $\ggg$ and $\widetilde{\ggg}$ repectively, thus $\widetilde{\bbb}^+=\bbb^+ \bigoplus \bbk(E_{nn}-E_{n+1,n+1})$.

Naturally regard $\chi$ as an element of ${\widetilde \ggg}^*$ with $\chi(\alpha_{n})=0$. Since $\chi$ is regular nilpotent of the form \eqref{eq: regular},
then $\chi$ is a subregular nilpotent $p$-character for $\widetilde \ggg$ of standard Levi form associated with  $I=\{\alpha_{1},\ldots,\alpha_{n-1}\}$. Correspondingly, Jantzen's arguments in \cite{Jan3} are available to $\widetilde\ggg$ because $p\nmid n+1$.

Let $R({\ggg})$, $R(\widetilde{\ggg})$ be the root systems of $\ggg, \widetilde{\ggg}$ respectively. Denote by $R^+(\ggg), R^+(\widetilde\ggg)$  the positive root sets in $R({\ggg}), R(\widetilde{\ggg})$ respectively. Denote by $W( \ggg)$ and $W(\widetilde \ggg)$ the corresponding Weyl groups, respectively. Since $W(\widetilde \ggg)_I$ is the subgroup of $W(\widetilde \ggg)$ generated by $s_\alpha$ with $\alpha \in I$, then $W(\widetilde \ggg)_I \cong W(\ggg) \cong \frak S_n$. Let $T$, $\widetilde{T}$ be the maximal tori of $\text{SL}_n, \text{SL}_{n+1}$ respectively, and denote by $X(T)$, $X(\widetilde{T})$ the character group of $T$, $\widetilde{T}$ respectively. Note that $W(\ggg)$ and $W(\widetilde \ggg)$  can also be regarded as the Weyl groups of $\text{SL}_n$ and $\text{SL}_{n+1}$ respectively. We first have
\begin{align}\label{eq: theta}
R^+(\widetilde\ggg)\backslash R^+(\ggg)=\{\theta_k:=\sum_{i=k}^n\alpha_{i}\mid k=1,\cdots,n\}.
\end{align}
Given $\lambda \in \sum_{i=1}^{n-1}\bbf_p \alpha_i \subsetneq \hhh^*$ and  $\mu \in \sum_{i=1}^{n}\bbf_p \alpha_i \subsetneq {\widetilde{\hhh}}^*$, denote by $Z_\chi(\lambda), \widetilde Z_\chi(\mu)$ and $L_\chi(\lambda), \widetilde L_\chi(\mu)$ the baby Verma modules and their simple quotients of $U_\chi(\ggg), U_\chi(\widetilde{\ggg})$, respectively. Given $\mu \in \sum_{i=1}^{n}\bbf_p \alpha_i \subsetneq {\widetilde{\hhh}}^*$, since $\chi(\alpha_{n})=0$ (Here $\chi \in {\widetilde \ggg}^*$), then we can define $\widetilde{Z_\chi}(\mu,\alpha_{n})$ as in (\ref{eq: levi verma}).  We have the following observations:

(1) In ${\ggg}=\frak {sl}_{n}$, we have ${\rho}:=\dfrac{1}{2}\sum_{\alpha \in R^+({\ggg})}\alpha=\sum_{i=1}^{n-1}\varpi_i=(n-1)\varepsilon_1+(n-2)\varepsilon_2+ \cdots +2\varepsilon_{n-2}+\varepsilon_{n-1}$.

(2) In $\widetilde{\ggg}=\frak {sl}_{n+1}$, we have $\widetilde{\rho}:=\dfrac{1}{2}\sum_{\alpha \in R^+(\widetilde{\ggg})}\alpha=\sum_{i=1}^{n}\varpi_i=n\varepsilon_1+(n-1)\varepsilon_2+ \cdots +2\varepsilon_{n-1}+\varepsilon_{n}$.

Set $H_n:=E_{n,n}-E_{n+1,n+1}$, $\iota_I:=\ggg \oplus\bbk H_n$, $\frak u_I:=\bigoplus_{\alpha \in R^+(\widetilde\ggg)\backslash R^+(\ggg)}\widetilde{\ggg}_\alpha$, then $\widetilde{\hhh}=\hhh \oplus \bbk H_n$. Similarly we can define $\frak u_I^-:=\bigoplus_{\alpha \in R^+(\widetilde\ggg)\backslash R^+(\ggg)}\widetilde{\ggg}_{-\alpha}$.
Denote by $N, \widetilde{N}$ the number of positive roots in $R({\ggg}), R(\widetilde{\ggg})$ respectively. Set $\widetilde{\bbb}:=\widetilde{\hhh} \oplus \sum_{\alpha \in R^+(\widetilde{\ggg})}\widetilde{\ggg}_\alpha$, $\widetilde{\ppp}:=\widetilde{\bbb} \oplus \widetilde{\ggg}_{-\alpha_{n}}$ where $X_{-\alpha_{n}}=E_{n+1,n}$, $\widetilde{\ggg}_{-\alpha_{n}}=\bbk X_{-\alpha_{n}}$ which is the weight space of $\widetilde{\ggg}$. We define $\frak p_I:=\iota_I \oplus \frak u_I$, then $\frak p_I$ becomes a standard parabolic subalgebra of $\widetilde{\ggg}$ associated with subsystem $\{\alpha_{1},\ldots,\alpha_{p-1}\}$. We have $\widetilde{\ggg}=\frak p_I \oplus \frak u_I^-$. Since $\chi(\ppp_I)=0$, then ${U_{\chi}(\frak p_I)}={U_0(\frak p_I)}$. All $L_\chi(\lambda), Z_\chi(\lambda)$ are $U_\chi(\ggg)$-modules which can be regarded as $U_0(\frak p_I)$-modules with the action as follows: Since $\lambda \in \hhh^*$, we can regard $\lambda$ as an element in $\widetilde{\hhh}^*$, thus each weight of $L_\chi(\lambda)$ can be viewed as an element in $\widetilde{\hhh}^*$. Let $v_\mu$ be the weight vector of $L_\chi(\lambda)$ with weight $\mu \in \widetilde{\hhh}^*$, then $\mu(H_n)=\mu(E_{nn})$. We define $H_n.v_\mu:=\mu(H_n)v_\mu$, and $\uuu_I.v_\mu:=0$ (i.e. the action of $\uuu_I$ is trivial). It is easy to verify that the action above makes $L_\chi(\lambda)$ to be $U_0(\frak p_I)$-modules.

Given $\lambda \in \sum_{i=1}^{n-1}\bbf_p \alpha_i \subsetneq \hhh^*$, we consider the following  $U_\chi(\widetilde{\ggg})$-module
$$\widetilde{L_\chi}(\lambda,\frak p_I):=U_\chi(\widetilde{\ggg}) \otimes_{U_{0}(\frak p_I)}L_\chi(\lambda).$$

Let $\overline{C(\widetilde T)}:=\{\lambda \in X(\widetilde{T}) \mid 0\leq \langle\lambda+\widetilde{\rho},\alpha^{\vee}\rangle \leq p, \forall \alpha \in R^+(\widetilde\ggg)\}$ which is the closure of the first dominant alcove.

Assume that $\lambda \in W(\ggg).0$ (Here $0 \in \hhh^*$), then $\lambda$ can be viewed as an element in $\widetilde{\hhh}^*$, hence $\lambda \in W(\widetilde{\ggg})_I.0$ (Here $0 \in \widetilde{\hhh}^*$). So Lemma \ref{lem: fundamental iso} implies that $Z_\chi(\lambda) \cong Z_\chi(0),L_\chi(\lambda) \cong L_\chi(0),\widetilde{Z_\chi}(\lambda) \cong \widetilde{Z_\chi}(0)$ and $\widetilde{L_\chi}(\lambda) \cong \widetilde{L_\chi}(0)$.

Let $u=\prod\limits_{i=1}^{n}X_{-\theta_i}^{p-1}$ with $\theta_i:=\sum\limits_{k=i}^n\alpha_{k}$ (Note that the result of $\prod\limits_{i=1}^{n}X_{-\theta_i}^{p-1}$ has nothing to do with the order of multiplications). Then for each $v \in \widetilde{L_\chi}({\lambda})$, we have $(X_{-\theta_i}^{p}-X_{-\theta_i}^{[p]}).v
=\chi^p(X_{-\theta_i})v$ which implies $X_{-\theta_i}^p.v=0$ (note that $X_{-\theta_i}^{[p]}=0$ and $\chi(X_{-\theta_i})=0$). So we can conclude that $u \widetilde{L_\chi}({\lambda})$ and $u \otimes L_\chi({\lambda})$ are $U_\chi(\ggg)$-modules, where $u \otimes L_\chi({\lambda}):=\{u \otimes x \mid x \in L_\chi({\lambda})\}$ is a non-zero submodule of $\widetilde{L_\chi}({\lambda},\frak p_I)$. Since Remark \ref{remark of Theorem 0.3} implies that $\dim L_\chi({\lambda})=p^{N-1}$, then we can easily see that $\dim(u \otimes L_\chi({\lambda}))=p^{N-1}$, so Lemma \ref{lem: 2.5.3} yields that $u \otimes L_\chi({\lambda})$ is irreducible as a $U_\chi(\ggg)$-module. So there exists $\nu \in \hhh^*$, such that $u \otimes L_\chi({\lambda}) \cong L_\chi(\nu)$ is a $U_\chi(\ggg)$-module isomorphism, hence $\text{dim}(L_\chi(\nu)) \leq p^{N-1}$ (Actually, $\text{dim}(L_\chi(\nu))=p^{N-1}$). By Remark \ref{remark of Theorem 0.3}, we conclude that $\nu \in W(\ggg).0$ (here $0 \in \hhh^*$). Thanks to Lemma \ref{lem: fundamental iso}, $\lambda,\nu \in W(\ggg).0$ yields that $L_\chi(\nu) \cong L_\chi(0) \cong L_\chi(\lambda)$ (as $U_\chi(\ggg)$-module isomorphisms). Thus $U_\chi(\ggg)$-module isomorphisms
\begin{align}\label{pre tensor iso}
	u \otimes L_\chi({\lambda}) \cong L_\chi({\lambda}) \cong L_\chi(0).
\end{align}

 So $u\widetilde{L_\chi}({\lambda},\frak p_I)=u(U_\chi(\widetilde{\ggg})\otimes_{U_{0}(\frak p_I)}L_\chi({\lambda}))=u U_\chi(\frak u_I^-) \otimes L_\chi({\lambda})=u \otimes L_\chi({\lambda})$. Hence we have $U_\chi(\ggg)$-module isomorphisms

\begin{align}\label{eq: pre iso}
u\widetilde{L_\chi}({\lambda},\frak p_I) \cong  L_\chi({\lambda}) \cong L_\chi(0).
\end{align}

By \cite[Theorem 2.6(a)]{Jan3}, we conclude that $\widetilde{L_\chi}({\lambda}) \cong \widetilde{Z_\chi}({\lambda},\alpha_{n})$ as $U_\chi(\widetilde{\ggg})$-modules and $\dim\widetilde{L_\chi}({\lambda}) \geq p^{\widetilde{N}-1}$. In particular, $\widetilde{L_\chi}(0) \cong \widetilde{Z_\chi}(0,\alpha_{n})$ and \begin{align}\label{>=}
\dim\widetilde{L_\chi}(0) \geq p^{\widetilde{N}-1}
\end{align}(In $\widetilde{L_\chi}({\lambda})$ and $\widetilde{L_\chi}(0)$, we view $\lambda,0 \in \hhh^*$; in $\widetilde{Z_\chi}({\lambda},\alpha_{n})$ and $\widetilde{Z_\chi}(0,\alpha_{n})$, we view $\lambda,0 \in \widetilde{\hhh}^*$). In fact, $\dim\widetilde{L_\chi}(0)=p^{\widetilde{N}-1}$ (see \cite[Section 1.5 Case 2]{Jan3}).


\begin{lemma}\label{lem: 2.5.2} There is a $U_\chi(\widetilde{\ggg})$-module isomorphism:  $\widetilde{L_\chi}(0,\frak p_I) \cong \widetilde{L_\chi}(0)$. This is to say $\widetilde{L_\chi}(0,\frak p_I)$ is a simple $U_\chi(\widetilde{\ggg})$-module.
\end{lemma}

\begin{proof}
	Remark \ref{remark of Theorem 0.3} implies that $\dim L_\chi(0)=p^{N-1}$. Since $\widetilde{L_\chi}(0,\frak p_I) \cong U_\chi(\frak u_I^-) \otimes_\bbk L_\chi(0)$ is a vector space isomorphism, then $\dim\widetilde{L_\chi}(0,\frak p_I)=\dim U_\chi(\frak u_I^-)\dim L_\chi(0)=p^{n+N-1}=p^{\widetilde{N}-1}$. Remark \ref{2.26} yields that $\widetilde{L_\chi}(0,\frak p_I)$ is a simple $U_\chi(\widetilde{\ggg})$-module. Clearly, we have a surjective $U_0(\ppp_I)$-module homomorphism $\phi:U_0(\frak p_I) \otimes_{U_0(\widetilde{\bbb})}\bbk_0 \rightarrow L_\chi(0)$. Since $\widetilde{L_\chi}(0,\frak p_I)=U_\chi(\widetilde{\ggg}) \otimes_{U_0(\frak p_I)}L_\chi(0)$ and $\widetilde{Z_\chi}(0)=U_\chi(\widetilde{\ggg}) \otimes_{U_0(\widetilde{\bbb})}\bbk_0 \cong U_\chi(\widetilde{\ggg}) \otimes_{U_0(\frak p_I)}U_0(\frak p_I) \otimes_{U_0(\widetilde{\bbb})}\bbk_0$ ($U_\chi(\widetilde{\ggg})$-module isomorphism), then $\phi:U_\chi(\frak p_I) \otimes_{U_0(\widetilde{\bbb})}\bbk_0 \rightarrow L_\chi(0)$ can induce a surjective $U_\chi(\widetilde{\ggg})$-module homomorphism $\widetilde{Z_\chi}(0) \rightarrow \widetilde{L_\chi}(0,\frak p_I)$. So $\widetilde{L_\chi}(0,\frak p_I)$ is a simple quotient of $\widetilde{Z_\chi}(0)$.
	The uniqueness of the simple quotient of $\widetilde{Z_\chi}(0)$ implies that $\widetilde{L_\chi}(0,\frak p_I) \cong \widetilde{L_\chi}(0)$ as $U_\chi(\widetilde{\ggg})$-modules.
\end{proof}

\subsubsection{The proof of Theorem \ref{Theorem 2.5.5}}\label{sec: proof of thm 0.3}
\begin{thm}\label{thm: 2.5.4} (Theorem \ref{Theorem 2.5.5}(1)) We have $U_\chi(\ggg)$-module isomorphisms $u \widetilde{L_\chi}(0) \cong L_\chi(0)$, so $L_\chi(0)$ can be described as $L_\chi(0) \cong u \widetilde{Z_\chi}(0,\alpha_{n})$ ($U_\chi(\ggg)$-module isomorphism).
\end{thm}

\begin{proof}
	Because $\widetilde{L_\chi}(0) \cong \widetilde{L_\chi}(0,\frak p_I)$, $\widetilde{L_\chi}(0) \cong \widetilde{Z_\chi}(0,\alpha_{n})$ are $U_\chi(\widetilde{\ggg})$-module isomorphisms, then $u \widetilde{L_\chi}(0)$ is a $U_\chi(\ggg)$-module yields that $u \widetilde{L_\chi}(0) \cong u\widetilde{L_\chi}(0,\frak p_I)$ and $u \widetilde{L_\chi}(0) \cong u \widetilde{Z_\chi}(0,\alpha_{n})$ are $U_\chi(\ggg)$-module isomorphisms. So take $\lambda=0$ in (\ref{eq: pre iso}), we have $U_\chi(\ggg)$-module isomorphisms $u \widetilde{L_\chi}(0) \cong u\widetilde{L_\chi}(0,\frak p_I) \cong L_\chi(0)$.
\end{proof}

Note that Theorem \ref{Theorem 2.5.5}(2) can be derived from Lemma \ref{lem: fundamental iso}, thus we have finished the proof of \ref{Theorem 2.5.5}.


In a brief words, we have
\begin{corollary}\label{final result} Let $\ggg=\frak {sl}_{n}$ with $p$ divides $n$ and $u=\prod\limits_{i=1}^{n}X_{-\theta_i}^{p-1}$ with $\theta_i:=\sum\limits_{k=i}^n\alpha_{k}$($1 \leq i \leq n$). For $\lambda \in \mathscr{X}_p=\sum_{i=1}^{n-1}\mathbb{F}_p\alpha_{i}$, the regular nilpotent irreducible modules are described as below
\begin{equation*}L_\chi(\lambda) \cong \begin{cases}
		Z_\chi(\lambda), & \text{if $\lambda \notin W.0$;}\\
		L_\chi(0) \cong u \widetilde{L_\chi}(0) \cong u \widetilde{Z_\chi}(0,\alpha_{n}), & \text{if $\lambda \in W.0$}.
	\end{cases}
\end{equation*}
where $W=W(\ggg)=\frak{S}_n$ is the Weyl group of $\ggg$.
\end{corollary}

\begin{proof}
	If $\lambda \notin W.0$, Theorem \ref{thm: 2.4.12} implies that $Z_\chi(\lambda) \cong L_\chi(\lambda)$ as $U_\chi({\ggg})$-modules.
	
	If $\lambda \in W.0$, by Theorem \ref{Theorem 2.5.5}, we have $u \widetilde{L_\chi}(0) \cong {L_\chi}(0) \cong {L_\chi}(\lambda)$ as $U_\chi({\ggg})$-modules. Thus yields the result.
\end{proof}

 \section{Irreducible modules of $\frak {sl}(np)$ for subregular nilpotent case}\label{sec: 3}

 In this section, let $\ggg=\frak {gl}_n(\bbk), \ggg'=\frak {sl}_n(\bbk)$. Let $T$ and $W$ be the maximal torus and Weyl group of $\GL_n(\bbk)$ respectively, denote by $R^+$ the set of positive roots in $\GL_n(\bbk)$ which can also be regarded as the set of positive roots of $\SL_n(\bbk)$.
Assume that $\chi \in {\ggg}^*$ is subregular nilpotent, and has form \eqref{eq: subregular}. This is to say, $\chi$ is a standard Levi form associated with $I=\{\alpha_{1},\alpha_{2},\cdots,\alpha_{n-2}\}$.
\subsection{Setting-up}
\subsubsection{} Recall that $W_I$ is the group generated by $s_{\alpha_i}$ with $1 \leq i \leq n-2$.  Obviously elements in $W_I$ all fix $\varepsilon_{n}$. Let $\sigma=s_1s_2\cdots s_{n-1}$, then $\sigma(i)=i+1$ for $1 \leq i \leq {n-1}$ and $\sigma(n)=1$. Thus $\sigma^j$, $j=0,1,\ldots,n-1$ are the representatives of  the right coset of $W_I$ in $W$.

Set $C=\{\mu\in X(T) \mid 0\leq\langle\mu+\rho,\alpha^{\vee}\rangle \leq p \mbox{ for all }\alpha\in R^+ \}$ be the closure of the first dominant alcove of $X(T)$.

Let $\lambda$ be any given element in $X(T)$. Then, there exists $\lambda_0 \in C$ and $w \in W$ such that $Z_\chi(\lambda) \cong Z_\chi(w.\lambda_0)$. Clearly, $w \in W$ implies $w \in W_I\sigma^i$ for some $0 \leq i \leq {n-1}$, thus $\lambda \in W_I\sigma^i.\lambda_0$. Since we have $Z_\chi(\mu) \cong Z_\chi(w.\mu)$ for all $w \in W_I$ and $\mu \in X(T)$ (see Lemma \ref{lem: fundamental iso}), then $Z_\chi(\lambda) \cong Z_\chi(w.\lambda_0) \cong Z_\chi(\sigma^i.\lambda_0)$. So it suffices to consider $Z_\chi(\sigma^i.\lambda_0)$.

 Now we already have $\lambda_0+\rho=\sum_{i=1}^{n-1}r_i\varpi_i$ with $r_i\geq 0$.
  Set $\lambda_i=\sigma^i.\lambda_0$ for $i=1,\ldots, n$,
  and $r_0=p-\sum_{i=1}^{n-1}r_i$. Let $m_i$ be the integer with $0< m_i\leq p$ and $m_i\equiv \langle \lambda_i+\rho, \alpha_{n-1}^{\vee}\rangle (\text{mod } p)$.

\subsubsection{$U_\chi(\ggg)$-$T_0$-modules $\widehat Z_\chi(\lambda_i, \alpha_{n-1})$}
 For $\psi=\sum_{i=1}^{n-1} i\alpha_i^\vee\in Y(T)$, take $T_0=\bigcap_{i=1}^{n-2} \ker(\alpha_i)$. Then $T_0$ is  a connected closed subgroup of $T$, which is actually the one parameter subgroup corresponding to $\psi$.

Following \cite[\S1.8]{Jan3}, we introduce the category of $U_\chi(\ggg)$-$T_0$-modules, whose object is by defintion a vector space over $\bbk$ that has a structure of both a $U_\chi(\ggg)$-module and as a $T_0$-module such that both structure are compatible, i.e.
$t(Xv)=\textsf{Ad}(t)(X)(tv)$ for all $t\in T_0$, $X\in\ggg$ and $v\in V$; and
the restriction of the $\ggg$-action on $V$ to $\Lie (T_0)$ is equal to the derivative of the $T_0$-aciton on $V$.

 Keep in mind that  $\chi(X_{-\alpha_{n-1}})=0$, and  $\chi$ is a standard Levi form with $I=\{\alpha_{1},\ldots,\alpha_{n-2}\}$. Thus we can define $Z_\chi(\lambda_i,\alpha_{n-1})$ ($0 \leq i \leq n-1$) as in \eqref{eq: levi verma}. Furthermore, this module is naturally endowed with $U_\chi(\ggg)$-$T_0$-module structure, which is denoted by $\widehat Z_\chi(\lambda_i, \alpha_{n-1})$. By definition, $\widehat Z_\chi(\lambda_i, \alpha_{n-1})$($0 \leq i \leq n-1$) has a basis of the form (see \cite[Section 1.5 Case 2]{Jan3} or see the last paragraph of \cite[Section 2.2]{Jan3})
 \begin{align}\label{eq: subregular basis}
 \prod_{\alpha>0,\alpha\neq\alpha_{n-1}} X_{-\alpha}^{m(\alpha)} \otimes (X_{-\alpha_{n-1}}^i \otimes 1_{\lambda_i})
 \end{align}
where $m(\alpha)$ runs from $0$ to $p-1$, and $i$ from $0$ to $m_i-1$.
Naturally, both $Z_\chi(\mu)$ and $L_\chi(\mu)$ for any $\mu\in X(T)$ are endowed with $U_\chi(\ggg)$-$T_0$-module structure, denoted by $\widehat Z_\chi(\mu)$ and $\widehat L_\chi(\mu)$, respectively.

Although $\ggg=\frak {gl}_n(\bbk)$ with $p \mid n$ is not the case in \cite{Jan3}. But we can also simulate the proof of \cite[Lemma 2.3 and Theorem 2.6]{Jan3} to conlude the following Theorem(see \cite[Remark 9.3]{GP}).

\begin{thm}\label{thm: 3.1.2}
(1) Assume that $0 \leq i\leq n-1$. For each $\widehat L_\chi(\lambda_j)$ with $0 \leq j \leq {n-1}$, we have $[\widehat Z_\chi(\lambda_i):\widehat L_\chi(\lambda_j)]=1$, both as $U_\chi(\ggg)$-modules, and as $U_\chi(\ggg)$-$T_0$-modules. It satisfies that $\text{dim}L_\chi(\lambda_j)=p^{N-1}r_{n-1-j}$, where $N=\sharp{R^+}$.  Furthermore, $\widehat L_\chi(\lambda_j) \cong \widehat Z_\chi(\lambda_j,\alpha_{n-1})$ for $0 \leq i \leq {n-1}$, and we have $[\widehat Z_\chi(\lambda_i)]=\sum_{r_{n-1-j}>0}[\widehat L_\chi(\lambda_j)]$.

(2) If $r_{n-i-1}=0$ with $0 \leq i \leq {n-2}$, then $\widehat Z_\chi(\lambda_i) \cong \widehat Z_\chi(\lambda_{i+1})$; if $r_0=0$, $\widehat Z_\chi(\lambda_{n-1}) \cong \widehat Z_\chi(\lambda_0)$.

(3) Moreover, the $U_\chi(\ggg)$-module $Z_\chi(\lambda_i)$ is uniserial (see \cite[Theorem 2.6(b)]{Jan3}).
\end{thm}

\begin{remark}
	Theorem \ref{thm: 3.1.2}(1) implies that the composition factor of $\widehat Z_\chi(\lambda_i)$ is of the form $\widehat L_\chi(\lambda_j)$($r_{n-j-1}>0$) has multiplicity 1, both as a $U_\chi(\ggg)$-module and as a $U_\chi(\ggg)$-$T_0$-module.
\end{remark}
\begin{remark}
	Let $G=\text{SL}_n(k)$ if $\text{char}(k)=p$ does not divide $n$, and $G=\text{GL}_n(k)$ if $\text{char}(k)=p>0$(no matter whether $\text{char}(k)=p$ divides $n$ or not). Let $\ggg=\text{Lie}(G)$, so that $\ggg=\frak {sl}_n(k)$, respectively $\ggg=\frak {gl}_n(k)$. Then $B: \ggg \times \ggg \rightarrow k$ with $B(x,y)=\text{tr}(xy)$ for $x,y \in \ggg$ is a non-degenerate $G$-invariant bilinear form. So we can conclude Theorem 3.1 in theses cases.
	
	However if $G=\text{SL}_n(k)$ with $\text{char}(k)=p$ divides $n$, then $B$
	is no longer non-degenerate. We do not have a non-degenerate $G$-invariant bilinear form on $\ggg=\text{Lie}(G)=\frak {sl}_n(k)$. So we can not get Theorem 3.1 in this case.
\end{remark}




\subsection{The restriction of $L_\chi(\lambda_i)$ to $U_\chi(\frak g')$}
\begin{thm}\label{thm: 3.1.3} ${\res}_\chi L_\chi(\lambda_i)$ is a simple $U_\chi(\frak {sl}_n)$-module.
\end{thm}

\begin{proof} Keep in mind that $\lambda_i+\rho=\sum_{j=1}^{n}a_j\varepsilon_j$ with $a_j \in \mathbb{F}_p$ and $W_I=\langle s_\alpha \mid \alpha \in I \rangle\cong\frak{S}_{n-1}$. So all permutations from $W_I$ preserve $\{n\}$. Hence, for each $\sigma \in W_I$, we have $\sigma(n)=n$ which implies that $a_{\sigma(n)} \neq a_n+1$. Since $\sigma^{-1}(\lambda_i+\rho)=\sum_{j=1}^{n}a_{\sigma(j)}\varepsilon_j$ $\lambda_i+\rho+\varpi_n=\sum_{j=1}^{n}(a_j+1)\varepsilon_j$, then $\lambda_i+\rho+\varpi_n \neq \sigma^{-1}(\lambda_i+\rho)$, namely $\sigma^{-1}.\lambda_i \neq \lambda_i+\varpi_n$. Due to Lemma \ref{lem: fundamental iso 2}, $L_\chi(\lambda_i) \ncong L_\chi(\lambda_i+\varpi_n)$. Lemma \ref{lem: 2.3.2} implies that $L_\chi(\lambda_i) \otimes k_1=L_\chi(\lambda_i+\varpi_n)$, thus $L_\chi(\lambda_i) \otimes k_1 \ncong L_\chi(\lambda_i+\varpi_n)$. So we conclude that $E \otimes k_1 \ncong E$. By Lemma \ref{lem: 2.2.3}(2), we have $E \otimes k_a \ncong E$ for all $a \in \mathbb{F}_p^\times$.

 Let $P_E$ be the projective cover of $E$ as $U_{\chi}(\frak {gl}_n)$-module. Corollary \ref{cor: 2.4.8} yields that there exists a simple $U_{\chi}(\frak {sl}_n)$-module $L$, such that ${\res}_{\chi}P_E$ is the projective cover of $L$.  Obviously, we have $\dim\Hom_{\frak {gl}_n}(P_E, E)=1$. Since $E$ is non-isomorphic to $E \otimes \bbk_a$ for all $a \in \mathbb{F}_p^\times$, then $\Hom_{\frak {gl}_n}(P_E, E \otimes \bbk_a)=0$ for all $a \in \mathbb{F}_p^\times$. Thus we have
 \begin{align*}
 [{\res}_{\chi}E: L]&\xlongequal[]{\text{Lemma }\ref{lem: 2.4.9}}{\dim\Hom}_{\frak {sl}_n}({\res}_{\chi}P_E, {\res}_{\chi}E)\xlongequal[]{(\ref{sum})}\bigoplus_{a\in \mathbb{F}_p}{\dim\Hom}_{\frak {sl}_n}({\res}_{\chi}P_E, {\res}_{\chi}E)_a\cr
 &\xlongequal[]{\text{Lemma }\ref{lem: 2.2.1}}\bigoplus_{a\in \mathbb{F}_p}{\dim\Hom}_{\frak {gl}_n}(P_E, E \otimes \bbk_a)=1.
 \end{align*}
 The theorem follows from  Corollary \ref{cor: 2.4.3}.
\end{proof}

\begin{remark}\label{res is simple}
For any $\lambda \in X(T)$, ${\res}_\chi L_\chi(\lambda)$ is a simple $U_\chi(\frak {sl}_n)$-module.
\end{remark}

\begin{remark} Theorem \ref{thm: 3.1.3} implies that ${\res}_\chi Z_\chi(\lambda_i)$ has compositon factors of the form ${\res}_\chi L_\chi(\lambda_j)$ with $0 \leq j \leq {n-1}$ and $r_{n-1-j}>0$, but the multiplicity may not always be 1. The multiplicity will be given in Corollary \ref{cor: 3.1.5}.
\end{remark}


\begin{corollary}\label{cor: 3.1.4} Let $\chi \in {\ggg}^*$ be of standard Levi form with $I=\{\alpha_{1},\ldots,\alpha_{n-2}\}$. Then simple $U_\chi({\ggg}')$-module ${\res}_\chi L_\chi(\lambda_{i})$ and ${\res}_\chi L_\chi(\lambda_{j})$ with $0 \leq i,j \leq n-1$ are isomorphic if and only if $L_\chi(\lambda_{i}) \cong L_\chi(\lambda_{j}) \otimes \bbk_a$ as $U_\chi(\ggg)$-module for some $a \in \mathbb{F}_p$.
\end{corollary}
\begin{proof}
	Let $L$ be a simple $U_\chi({\ggg}')$-module which is a composition factor of $\text{res}_\chi L_\chi(\lambda_{j})$.
	First we consider the composition factors of $\text{ind}_\chi L$.
	
	Case 1. If for all $a \in \mathbb{F}_p$, we have $L_\chi(\lambda_{j}) \otimes \bbk_a \ncong L_\chi(\lambda_{j})$, namely $L_\chi(\lambda_{j}), L_\chi(\lambda_{j}) \otimes \bbk_1,\ldots,L_\chi(\lambda_{j}) \otimes \bbk_{p-1}$ are not isomorphic each other. By the proof of Theorem \ref{thm: 2.4.12} we have $\res_\chi L_\chi(\lambda_{j})$ is simple $U_\chi({\ggg}')$-module, i.e. $\res_\chi L_\chi(\lambda_{j}) \cong L$. By the reciprocity \eqref{eq: 2.1}),
\begin{align*}
\Hom_{\ggg}({\ind}_\chi L, L_\chi(\lambda_{j}) \otimes \bbk_a)
&\cong \Hom_{{\ggg}'}(L, {\res}_\chi (L_\chi(\lambda_{j}) \otimes \bbk_a)) \cr
&\cong \Hom_{{\ggg}'}({\res}_\chi L_\chi(\lambda_{j}), {\res}_\chi L_\chi(\lambda_{j})) \neq 0.
\end{align*}
 So each $L_\chi(\lambda_{i}) \otimes \bbk_a$ with $a \in \mathbb{F}_p$ is a composition factor of ${\ind}_\chi L$. So there exists an $U_\chi(\ggg)$-module $M$, such that $[{\ind}_\chi L]=\sum_{a \in \mathbb{F}_p}[L_\chi(\lambda_{j}) \otimes \bbk_a]+[M]$, thus $\dim ({\ind}_\chi L)=\dim(\bigoplus_{a \in \mathbb{F}_p}L_\chi(\lambda_{j}) \otimes \bbk_a)+\dim M$.
  Since $\res_\chi L_\chi(\lambda_{j}) \cong L$, then we have
  $$\dim(\bigoplus_{a \in \mathbb{F}_p}L_\chi(\lambda_{j}) \otimes \bbk_a)=p\dim L_\chi(\lambda_{j})=p{\dim}({\res}_\chi L_\chi(\lambda_{j}))=p\dim L={\dim}({\ind}_\chi L).$$
  This yields that $\text{dim}M=0$, i.e. $M=0$. Then we have $[{\ind}_\chi L]=\sum_{a \in \mathbb{F}_p}[L_\chi(\lambda_{j}) \otimes \bbk_a]$.
	
	Case 2. If there exists $a \in \mathbb{F}_p^\times$, such that $L_\chi(\lambda_{j}) \otimes \bbk_a \cong L_\chi(\lambda_{j})$, then for all $a \in \mathbb{F}_p$, we have $L_\chi(\lambda_{j}) \otimes \bbk_a \cong L_\chi(\lambda_{j})$(by Lemma \ref{lem: 2.2.3}(2)). The Remark of Corollary \ref{cor: 2.4.3} shows that $L_\chi(\lambda_{j})$ is a composition factor of ${\ind}_\chi L$, then there exists an $U_\chi(\ggg)$-module $M$, such that $[\ind_\chi L]=[M]+[L_\chi(\lambda_{j})]$, thus $\dim (\ind_\chi L)=\dim M+\dim L_\chi(\lambda_{j})$. Since $L_\chi(\lambda_{j}), L_\chi(\lambda_{j}) \otimes k_1,...,L_\chi(\lambda_{j}) \otimes k_{p-1}$ are isomorphic each other, then $[\res_\chi L_\chi(\lambda_{j})]=p[L]$ by immitation of the proof of Theorem \ref{thm: 2.4.12}), hence $\dim L_\chi(\lambda_{j})=\dim (\res_\chi L_\chi(\lambda_{j}))=p \cdot \dim L=\dim(\ind_\chi L)$, namely $\dim L_\chi(\lambda_{j})=\dim(\ind_\chi L)$. So we can conclude that $\dim M=0$(i.e. $M=0$), thus $[\text{ind}_\chi L]=[L_\chi(\lambda_{j})]$. This yields that $\text{ind}_\chi L=L_\chi(\lambda_{j})$.
	
	The arguments in Case 1 and Case 2 show that the composition factor of ${\ind}_\chi L$ must be of the form $L_\chi(\lambda_{j}) \otimes \bbk_a$ for some $a \in \mathbb{F}_p$. Remark 2.12 of Corollary \ref{cor: 2.4.3} implies that $L$ is a composition factor of ${\res}_\chi L_\chi(\lambda_{i})$ if and only if $L_\chi(\lambda_{i})$ is a composition factor of ${\ind}_\chi L$. Hence, $L_\chi(\lambda_{i})$ must be written as $L_\chi(\lambda_{j}) \otimes \bbk_a$ for some $a \in \mathbb{F}_p$.
\end{proof}

\begin{remark}\label{Rem of iso} Assume $\frak S_n$ is the symmetric group with $n$ letters, $\text{Stab}_{\frak S_n}(n)$ means the stabiliser of $n$ in $\frak S_n$, i.e. $\forall \eta \in \text{Stab}_{\frak S_n}(n)$, we have $\eta(n)=n$.
	
	Using Lemma \ref{lem: 2.3.2}, we get $L_\chi(\lambda_{i}) \cong L_\chi(\lambda_{j}) \otimes k_a$ if and only if there exists $w \in W_I$ such that $\lambda_{j}+a\varpi_n+\rho=w(\lambda_{i}+\rho)$, i.e. $\lambda_{j}+a\varpi_n+\rho \in W_I.(\lambda_{i}+\rho)$. Set $\lambda_0+\rho=\sum_{l=1}^{n}t_l\varepsilon_l$ with $t_l \in \mathbb{F}_p$(Note: $\lambda_0+\rho \in X(T)$, namely $\lambda_0+\rho \in \Lambda$), then $\lambda_{i}+\rho=\sigma^i(\lambda_0+\rho)=\sum_{l=1}^{n}t_l\varepsilon_{l+i}=\sum_{l=1}^{n}t_{l-i}\varepsilon_{l}$, $\lambda_{j}+\rho=\sigma^j(\lambda_0+\rho)=\sum_{l=1}^{n}t_l\varepsilon_{l+j}=\sum_{l=1}^{n}t_{l-j}\varepsilon_{l}$, thus $\lambda_{j}+\rho+a\varpi_n=\sum_{l=1}^{n}(t_l+a)\varepsilon_{l+j}=\sum_{l=1}^{n}(t_{l-j}+a)\varepsilon_{l}$  (Note: If $k>n$, then there exists $k'$ with $1 \leq {k}' \leq n$ such that ${k}' \equiv k$(mod $n$). In this case, we consider $\varepsilon_k, t_k$ as $\varepsilon_{k}',t_{k}'$).
	
	Since $\lambda_{j}+a\varpi_n+\rho \in W_I.(\lambda_{i}+\rho)$, then there exists $\tau \in \text{Stab}_{\frak S_n}(n) \subsetneq \frak S_n$, such that $\sum_{l=1}^{n}(t_{l-j}+a)\varepsilon_{l}=\sum_{l=1}^{n}t_{l-i}\varepsilon_{\tau(l)}$. It is easy to see $\sum_{l=1}^{n}(t_{l-j}+a)\varepsilon_{l}=\sum_{l=1}^{n}(t_{\tau(l)-j}+a)\varepsilon_{\tau(l)}$, thus $\sum_{l=1}^{n}(t_{\tau(l)-j}+a)\varepsilon_{\tau(l)}=\sum_{l=1}^{n}t_{l-i}\varepsilon_{\tau(l)}$. This implies that $t_{l-i}-t_{\tau(l)-j}=a$ for all $1 \leq l \leq n-1$. Since $\tau \in Stab_{S_n}(n)$, $\tau(n)=n$. So we have $a=t_{n-i}-t_{\tau(n)-j}=t_{n-i}-t_{n-j}$. This yields that $t_{l-i}-t_{\tau(l)-j}=t_{n-i}-t_{n-j}$ for all $1 \leq l \leq n-1$.
\end{remark}

From now on, let $\hhh$, $\hhh'$ be the cartan subalgebras of $\frak {gl}_n$, $\frak {sl}_n$ respectively; let $T$, $T'$ be the maximal tori of $\text{GL}_n$, $\text{SL}_n$ respectively.

For $\lambda \in X(T')$, we can regard $\lambda$ as an element in $X(T)$ which is denoted by $\underline{\lambda}$, thus $\underline{\lambda}|_{T'}=\lambda$. It is easy to see that $d\lambda \in X(T')/pX(T') \cong \sum_{i=1}^{n-1}\mathbb{F}_p\alpha_i$ and $d\underline{\lambda} \in X(T)/pX(T) \cong \sum_{i=1}^{n}\mathbb{F}_p\varepsilon_i$. Let $C'=\{\mu\in X(T') \mid 0\leq\langle\mu+\rho,\alpha^{\vee}\rangle \leq p \mbox{ for all }\alpha\in R^+ \}$ be the closure of the first dominant alcove of $X(T')$.

For each $\lambda \in X(T')$, there exists $\lambda_0 \in C'$ and $w \in W$ such that $\lambda=w.\lambda_0$, i.e. $\lambda \in W.\lambda_0$. We set $\lambda_i=\sigma^i.\lambda_0$($1 \leq i \leq  n-1$). For $\chi \in \frak {sl}^*_n$, we can extend $\chi$ to the element in $\frak {gl}^*_n$ which is denoted by $\underline{\chi}$, thus $\underline{\chi}|_{\frak {sl}_n}=\chi$.

Given $\lambda \in X(T')$ and $\chi \in \frak {sl}^*_n$ which is subregular nilpotent, we can define $U_\chi(\frak {sl}_n)$-modules $Z_\chi(\lambda)$, $L_\chi(\lambda)$ and $Z_\chi(\lambda,\alpha_{n-1})$. Since $\underline{\lambda} \in X(T)$ and $\underline{\chi} \in \frak {gl}^*_n$, we can define $U_{\underline{\chi}}(\frak {gl}_n)$-modules $Z_{\underline{\chi}}(\underline{\lambda})$, $L_{\underline{\chi}}(\underline{\lambda})$ and $Z_{\underline{\chi}}(\underline{\lambda},\alpha_{n-1})$. For convenience, we set $\underline{Z_\chi}(\lambda):=Z_{\underline{\chi}}(\underline{\lambda})$, $\underline{L_\chi}(\lambda):=L_{\underline{\chi}}(\underline{\lambda})$ and $\underline{Z_\chi}(\lambda,\alpha_{n-1}):=Z_{\underline{\chi}}(\underline{\lambda},\alpha_{n-1})$.

Since there exists $\lambda_0 \in C'$ and $w \in W$ such that $\lambda=w.\lambda_0$, then $w \in W_I.\sigma^i$ for some $0 \leq i \leq {n-1}$ which implies $\lambda \in W_I\sigma^i.\lambda_0$. Since we have an $U_\chi(\frak {sl}_n)$-module isomorphism $Z_\chi(\mu) \cong Z_\chi(w.\mu)$ for all $w \in W_I$ and $\mu \in X(T')$, then $Z_\chi(\lambda)=Z_\chi(w.\lambda_0) \cong Z_\chi(\sigma^i.\lambda_0)$(as $U_\chi(\frak {sl}_n)$-module isomorphism). So it suffices to consider $Z_\chi(\lambda_{i})=Z_\chi(\sigma^i.\lambda_0)$.

For $\lambda \in X(T')$, we can still use $\lambda$, $\underline{\lambda}$ to denote $d\lambda$, $d\underline{\lambda}$ respectively, if the context is clear. Since $d\lambda \in \mathscr X_p:=\bigoplus_{i=1}^{n-1}\mathbb{F}_p\alpha_i$, thus $\lambda$ can be viewed as an element in $\mathscr X_p$. By the definition of $\lambda_{i}$($0 \leq i \leq n-1$), we have $\lambda_{i} \in X(T')$($0 \leq i \leq n-1$), hence $\lambda_{i} \in X(T')$($0 \leq i \leq n-1$) can be regarded as an element in $\mathscr X_p$ (Here $\mathscr X_p$ is the same as the $\mathscr X_p$ defined in \S\ref{X_p} for $\ggg=\frak {sl}_n$).

\begin{lemma}\label{simple res}
	For any $\lambda \in \Lambda=\bigoplus_{i=1}^n \mathbb{F}_p \varepsilon_{i}$, then $\res_{\chi}L_\chi(\lambda) \cong L_{\chi|_{\ggg'}}(\lambda|_{\hhh'})$ as $U_\chi(\frak {sl}_n)$-modules.
\end{lemma}

\begin{proof}
For $\lambda \in \Lambda$, Remark \ref{res is simple} implies that $\res_{\chi}L_\chi(\lambda)$ is a simple $U_\chi(\frak {sl}_n)$-module. Lemma \ref{lem: 2.4.10}(3) implies that  $\res_{\chi}Z_\chi(\lambda) \cong Z_{\chi|_{\ggg'}}(\lambda|_{\hhh'})$.

Let $N$ be the unique maximal submodule of the $U_\chi(\ggg)$-module $Z_\chi(\lambda)$, and $M$ be the unique maximal submodule of the $U_\chi(\ggg')$-module $\res_{\chi}Z_\chi(\lambda) \cong Z_{\chi|_{\ggg'}}(\lambda|_{\hhh'})$. Clearly, $N \subseteq M$, i.e. $N$ is the submodule of $M$ as $U_\chi(\ggg')$-modules.

Since $L_\chi(\lambda) \cong Z_\chi(\lambda)/N$ as $U_\chi(\ggg)$-modules, then we have $U_\chi(\ggg')$-module isomorphisms $\res_\chi{L_\chi(\lambda)} \cong \res_\chi{(Z_\chi(\lambda)/N)}=Z_{\chi|_{\ggg'}}(\lambda|_{\hhh'})/({\res_\chi N}) \stackrel{\psi}{\longrightarrow} Z_{\chi|_{\ggg'}}(\lambda|_{\hhh'})/M \cong L_{\chi|_{\ggg'}}(\lambda|_{\hhh'})$. Obviously, $\psi$ is a surjective $U_\chi(\ggg')$-module homomorphism, and $\ker \psi$ is a submodule of $\res_\chi{L_\chi(\lambda)}$. Since $\res_\chi{L_\chi(\lambda)}$ is a simple $U_\chi(\ggg')$-module, then $\ker \psi=0$ which implies that $\psi$ is injective, thus $\psi$ is a $U_\chi(\ggg')$-module isomorphism.
\end{proof}

Lemma \ref{simple res} implies that $L_\chi(\lambda_i) \cong \text{res}_{\underline{\chi}}{L_{\underline{\chi}}}(\underline{\lambda_i})$ as $U_\chi(\frak {sl}_n)$-modules with $0 \leq i \leq n-1$.

Because $\lambda_0 \in C'$, it is easy to see that $\underline{\lambda_0}\in C$. So  $Z_{\underline{\chi}}(\underline{\lambda_0})$ is the case we have just discussed.
Since $\underline{\lambda_0}+\rho \in \Lambda$(i.e. $\lambda_0+\rho \in X(T)$), set $\underline{\lambda_0}+\rho=\sum_{i=1}^{n}t_i\varepsilon_i$ with $t_i \in \mathbb{F}_p$.

\begin{corollary} \label{cor: 3.1.5} Assume that $\chi \in \frak {sl}^*_n$ is subregular nilpotent $p$-character of the form as in \eqref{eq: subregular}, and $\lambda_{i} \in \mathscr X_p$ ($0 \leq i \leq n-1$) and $t_i$ ($1 \leqslant i \leqslant n$) is defined above. Let $N_j:=\sharp \{0 \leq s \leq n-1 \mid \exists \tau \in \text{Stab}_{\frak S_n}(n), \text{such that }t_{l-j}-t_{\tau(l)-s}=t_{n-j}-t_{n-s} \text{ for all } 1 \leq l \leq n-1 , \text{and }r_{n-s-1}>0\}$ with $0 \leq j \leq n-1$. For $0 \leq i \leq n-1$, we have $[Z_\chi(\lambda_{i}):L_\chi(\lambda_{j})]=N_j$.
\end{corollary}

\begin{proof}
	For $\lambda_{i} \in \mathscr X_p$ with $0 \leq i \leq n-1$, we can extend $\lambda_i$ and $\chi$ to the elements in $\Lambda$ and $\frak {gl}^*_n$ respectively, and denote them by $\underline{\lambda_i}$ and $\underline{\chi}$. Thus $\underline{\lambda_i}|_{\hhh'}=\lambda_i$ and $\underline{\chi}|_{\frak {sl}_n}=\chi$. Theorem \ref{thm: 3.1.3} shows that $L_\chi(\lambda_i) \cong \res_{\underline{\chi}}{L_{\underline{\chi}}}(\underline{\lambda_i})$ with $0 \leq i \leq n-1$.
	
	Lemma \ref{lem: 2.4.10}(3) implies that $Z_\chi(\lambda_{i})=\res_{\underline{\chi}}{Z_{\underline{\chi}}}(\underline{\lambda_i})$. By Theorem \ref{thm: 3.1.2}(1), we obtain $[Z_{\underline{\chi}}(\underline{\lambda_i})]=\sum_{r_{n-1-j}>0}[L_{\underline{\chi}}(\underline{\lambda_j})]$.
	Thus $$[Z_\chi(\lambda_{i})]=[\res_{\underline{\chi}}{Z_{\underline{\chi}}}(\underline{\lambda_i})]=\sum_{r_{n-1-j}>0}[\res_{\underline{\chi}}L_{\underline{\chi}}(\underline{\lambda_j})].$$
	
	Due to Remark \ref{Rem of iso} and Corollary \ref{cor: 3.1.4}, we can conclude that $[Z_\chi(\lambda_{i}):\res_{\underline{\chi}}L_{\underline{\chi}}(\underline{\lambda_j})]=N_j$. Since $L_\chi(\lambda_i) \cong \res_{\underline{\chi}}{L_{\underline{\chi}}}(\underline{\lambda_i})$, then $[Z_\chi(\lambda_{i}):L_\chi(\lambda_{j})]=N_j$.
\end{proof}

\subsection{The proof of Theorem \ref{thm:subregular}}\label{sec: proof of thm 0.5}
Keep in mind that $\chi \in  {\ggg}^*$ is of standard Levi form with $I=\{\alpha_{1},...,\alpha_{n-2}\}$ and $\lambda_{i}$ defined above($\lambda_{i} \in \mathscr X_p$), we have $\underline{\chi} \in \frak {gl}_n^*$ and $\underline{\lambda_{i}} \in \Lambda$. Thus $\underline{\lambda_{i}}|_{\hhh'}=\lambda_{i}$ and $\underline{\chi}|_{\frak {sl}_n}=\chi$. Note that $Z_\chi(\lambda_i),L_\chi(\lambda_i),Z_{{\chi}}({\lambda_i},\alpha_{n-1})$ are $U_{{\chi}}(\frak {sl}_n)$-modules, $Z_{\underline{\chi}}(\underline{\lambda_i})$, $L_{\underline{\chi}}(\underline{\lambda_i})$ and $Z_{\underline{\chi}}(\underline{\lambda_i},\alpha_{n-1})$ are $U_{\underline{\chi}}(\frak {gl}_n)$-modules. It is east to see that $\text{dim}Z_{\underline{\chi}}(\underline{\lambda_i},\alpha_{n-1})=\text{dim}Z_{{\chi}}({\lambda_i},\alpha_{n-1})$.

Theorem \ref{thm: 3.1.2} implies that $\dim Z_{\underline{\chi}}(\underline{\lambda_i},\alpha_{n-1})=\dim L_{\underline{\chi}}(\underline{\lambda_i})$. Theorem \ref{thm: 3.1.3} shows that $L_\chi(\lambda_i) \cong \res_{\underline{\chi}}{L_{\underline{\chi}}}(\underline{\lambda_i})$. Thus we conclude that $\dim L_\chi(\lambda_i)=\dim Z_{{\chi}}({\lambda_i},\alpha_{n-1})$. It is clear that $L_\chi(\lambda_i)$ is a homomorphic image of $Z_{{\chi}}({\lambda_i},\alpha_{n-1})$. Thus $\dim L_\chi(\lambda_i)=\dim Z_{{\chi}}({\lambda_i},\alpha_{n-1})$ yields that $L_\chi(\lambda_{i}) \cong Z_\chi(\lambda_{i},\alpha_{n-1})$ as $U_{{\chi}}(\frak {sl}_n)$-modules. It is clear that $Z_\chi(\lambda_{i},\alpha_{n-1})$ is spanned by a basis in \eqref{eq: subregular basis}. By Theorem \ref{thm: 3.1.2}, we have $\dim Z_{\underline{\chi}}(\underline{\lambda_i},\alpha_{n-1})=\dim L_{\underline{\chi}}(\underline{\lambda_{i}})=p^{N-1}r_{n-1-i}$. Since $\dim Z_{\underline{\chi}}(\underline{\lambda_i},\alpha_{n-1})=\dim Z_{{\chi}}({\lambda_i},\alpha_{n-1})$  and $\dim L_\chi(\lambda_i)=\dim Z_{{\chi}}({\lambda_i},\alpha_{n-1})$, then $\dim L_\chi(\lambda_{i})=p^{N-1}r_{n-1-i}$. So Theorem \ref{thm:subregular}(1) holds. Theorem \ref{thm:subregular}(2) is Corollary \ref{cor: 3.1.5}.

\subsection{An example}

Now we use Corollary \ref{cor: 3.1.5} to handle the case $n=p=3$.

\begin{example}
Let $\chi \in \frak {sl}_3^*$ is subregular nilpotent and be of standard Levi form with $I=\{\alpha_{1}\}$, and $p=3, \sigma=s_1s_2$ ($n=p=3$ in Theorem \ref{thm:subregular}). Since $\lambda$ is determined  by $\lambda(h_{\alpha_1})$ and $\lambda(h_{\alpha_2})$, we can regard $\lambda=(\lambda(h_{\alpha_1}),\lambda(h_{\alpha_2})) \in \mathbb{F}_3^2$. Thus $\rho=\varpi_1+\varpi_2=\alpha_{1}+\alpha_2=(1,1)$. So there are 4 types of $W$-orbits:

(a) $\lambda_1^1=(0,0), \lambda_2^1=(1,1)$. Note that $\sigma^2.\lambda_1^1=\sigma.\lambda_1^1=\lambda_1^1$, $\lambda_2^1 \in W_I.\lambda_1^1$;

(b) $\lambda_1^2=(2,2)$. Note that $\sigma^2.\lambda_1^2=\sigma.\lambda_1^2=\lambda_1^2$;

(c) $\lambda_1^3=(1,0), \lambda_2^3=(0,2), \lambda_3^3=(2,1)$. Note that $\sigma.\lambda_1^3=\lambda_3^3, \sigma^2.\lambda_1^3=\lambda_2^3, \lambda_2^3 \in W_I.\lambda_1^3$;

(d) $\lambda_1^4=(0,1), \lambda_2^4=(1,2), \lambda_3^4=(2,0)$. Note that $\sigma.\lambda_1^4=\lambda_3^4, \sigma^2.\lambda_1^4=\lambda_2^4, \lambda_2^4 \in W_I.\lambda_1^4$.

For case (a), since $\lambda_2^1 \in W_I.\lambda_1^1$, then $Z_\chi(\lambda_1^1) \cong Z_\chi(\lambda_2^1)$(Lemma \ref{lem: fundamental iso}). Since $\underline{\lambda_1^1}+\rho=\varpi_1+\varpi_2$, thus we have $r_1=r_2=1$ and $r_0=3-r_1-r_2=1$. By Theorem \ref{thm: 3.1.2}(1), we have $[Z_{\underline{\chi}}(\underline{\lambda_1^1})]=[L_{\underline{\chi}}(\underline{\lambda_1^1})]+[L_{\underline{\chi}}(\sigma.\underline{\lambda_1^1})]+[L_{\underline{\chi}}(\sigma^2.\underline{\lambda_1^1})]$(as $U_\chi(\frak {gl}_3)$-modules).
Theorem \ref{thm: 3.1.3} deduces that $\text{res}_{\underline{\chi}}L_{\underline{\chi}}(\sigma^i.\underline{\lambda_1^1}) \cong L_\chi(\sigma^i.\lambda_1^1)$ with $i=0,1,2$, and Lemma \ref{lem: 2.4.10}(3) implies that $Z_\chi(\lambda_1^1)=\text{res}_{\underline{\chi}}(Z_{\underline{\chi}}(\underline{\lambda_1^1}))$. Hence $[Z_\chi(\lambda_1^1)]=[\text{res}_{\underline{\chi}}(Z_{\underline{\chi}}(\underline{\lambda_1^1}))]=[\text{res}_{\underline{\chi}}(L_{\underline{\chi}}(\underline{\lambda_1^1}))]+[\text{res}_{\underline{\chi}}(L_{\underline{\chi}}(\sigma.\underline{\lambda_1^1}))]+[\text{res}_{\underline{\chi}}(L_{\underline{\chi}}(\sigma^2.\underline{\lambda_1^1}))]=[L_\chi(\lambda_1^1)]+[L_\chi(\sigma.\lambda_1^1)]+[L_\chi(\sigma^2.\lambda_1^1)]=3[L_\chi(\lambda_1^1)]$.

For case (b), $\underline{\lambda_1^2}=(1,2,0)=\varepsilon_1+2\varepsilon_2$, then $\lambda_1^2+\rho=0$ which implies $r_1=r_2=0, r_0=3-r_1-r_2=3$. Using Theorem \ref{thm: 3.1.2}(1), we get $Z_{\underline{\chi}}(\underline{\lambda_1^2}) \cong L_{\underline{\chi}}(\underline{\lambda_1^2})$($U_\chi(\frak {gl}_3)$-module isomorphism), thus $\text{res}_{\underline{\chi}} Z_{\underline{\chi}}(\underline{\lambda_1^2}) \cong \text{res}_{\underline{\chi}} L_{\underline{\chi}}(\underline{\lambda_1^2})$.
 By Lemma \ref{lem: 2.4.10}(3), we have $\text{res}_{\underline{\chi}}(Z_{\underline{\chi}}(\underline{\lambda_1^2}))=Z_\chi(\lambda_1^2)$. Theorem \ref{thm: 3.1.3} implies that $\text{res}_{\underline{\chi}}(L_{\underline{\chi}}(\underline{\lambda_1^2})) \cong L_\chi(\lambda_1^2)$. Hence $Z_\chi(\lambda_1^2)=\text{res}_{\underline{\chi}} Z_{\underline{\chi}}(\underline{\lambda_1^2}) \cong \text{res}_{\underline{\chi}} L_{\underline{\chi}}(\underline{\lambda_1^2})\cong L_\chi(\lambda_1^2)$, namely $Z_\chi(\lambda_1^2) \cong L_\chi(\lambda_1^2)$.

For case (c), since $\lambda_2^3 \in W_I.\lambda_1^3$, then $Z_\chi(\lambda_1^3) \cong Z_\chi(\lambda_2^3)$(Lemma \ref{lem: fundamental iso}). Due to the fact that $\underline{\lambda_1^3}=\varepsilon_1$, thus $\underline{\lambda_1^3}+\rho=\varepsilon_2=2\varpi_1+\varpi_2$, so $r_1=2,r_2=1,r_0=3-r_1-r_2=0$. By Theorem \ref{thm: 3.1.2}(1), we conclude $[Z_{\underline{\chi}}(\underline{\lambda_1^3})]=[L_{\underline{\chi}}(\underline{\lambda_1^3})]+[L_{\underline{\chi}}(\sigma.\underline{\lambda_1^3})]$. Since $\underline{\lambda_1^3}+\rho=\varepsilon_2$, then $t_1=0, t_{-1}=t_2=1, t_0=t_3=0$. We want to calculate $M_0$ and $N_0$ defined in Corollary \ref{cor: 3.1.5}.

If there exists $\tau \in \text{Stab}_{\frak S_3}(3)$ such that $t_{l}-t_{\tau(l)-s}=t_3-t_{3-s}$ for $l=1,2$.

If $s=1$, then
\begin{equation*}
\begin{cases}
t_{1}-t_{\tau(1)-1}=t_3-t_2 \\
t_{2}-t_{\tau(2)-1}=t_3-t_2 \\
\end{cases}
\end{equation*}

which implies
\begin{equation*}
\begin{cases}
t_{\tau(1)-1}=1 \\
t_{\tau(2)-1}=2 \\
\end{cases}
\end{equation*}

Since $t_0=t_1=0$ and $t_2=1$, then $t_{\tau(2)-1}=2$ can not be one of $t_0,t_1$ and $t_2$. This means such $\tau$ is non-existent, i.e. $1 \notin M_0$.

If $s=2$, then $r_{n-s-1}=r_0=0$. By the definition of $M_0$, we have $2 \notin M_0$.

If $s=0$, then $t_{l}-t_{\tau(l)-s}=t_3-t_{3-s}$ yields that $\tau(l)=l$. Since $l=1,2$, then $\tau(1)=1$, $\tau(2)=2$. So we must have $\tau(0)=0$, i.e. $\tau=\text{id}$. This means $0 \in M_0$.

Through the above discussion, we conclude that $M_0=\{0\}$, thus $N_0=1$.

$[Z_{\underline{\chi}}(\underline{\lambda_1^3})]=[L_{\underline{\chi}}(\underline{\lambda_1^3})]+[L_{\underline{\chi}}(\sigma.\underline{\lambda_1^3})]$ implies that $[Z_{{\chi}}({\lambda_1^3})]=[L_{{\chi}}({\lambda_1^3})]+[L_{{\chi}}({\lambda_3^3})]$. Since $N_0=1$, we have $[Z_{{\chi}}({\lambda_1^3}):L_{{\chi}}({\lambda_1^3})]=1$(Corollary \ref{cor: 3.1.5}). If $L_\chi(\lambda_1^3) \cong L_\chi(\lambda_3^3)$, then $[Z_{{\chi}}({\lambda_1^3})]=[L_{{\chi}}({\lambda_1^3})]+[L_{{\chi}}({\lambda_3^3})]$ implies that $[Z_{{\chi}}({\lambda_1^3}):L_{{\chi}}({\lambda_1^3})]=2$. This contradicts with $[Z_{{\chi}}({\lambda_1^3}):L_{{\chi}}({\lambda_1^3})]=1$. Hence this yields that $L_\chi(\lambda_1^3) \ncong L_\chi(\lambda_3^3)$ as $U_\chi(\frak {sl}_3)$-module.

For case (d), we also have $Z_\chi(\lambda_1^4) \cong Z_\chi(\lambda_2^4)$ and $[Z_\chi(\lambda_1^4)]=[L_\chi(\lambda_1^4)]+[L_\chi(\lambda_3^4)]$ with $L_\chi(\lambda_1^4) \ncong L_\chi(\lambda_3^4)$.
\end{example}


\begin{corollary}\label{cor: 3.2.1} Let $\bbk$ is algebraically closed field of characteristic 3, and $\chi \in \frak {sl}^*_3(\bbk)$ be of standard Levi form with $I=\{\alpha_{1}\}$ which is subregular nilpotent. Denote $\lambda_1^1=(0,0), \lambda_2^1=(1,1); \lambda_1^2=(2,2); \lambda_1^3=(1,0), \lambda_2^3=(0,2), \lambda_3^3=(2,1); \lambda_1^4=(0,1), \lambda_2^4=(1,2), \lambda_3^4=(2,0)$ where these $\lambda_{k}^j \in \mathbb{F}_3^2$.

Then $Z_\chi(\lambda_1^1) \cong Z_\chi(\lambda_2^1), [Z_\chi(\lambda_1^1)]=3[L_\chi(\lambda_1^1)]$;

$Z_\chi(\lambda_1^2) \cong L_\chi(\lambda_1^2)$;

$Z_\chi(\lambda_1^i) \cong Z_\chi(\lambda_2^i), [Z_\chi(\lambda_1^i)]=[Z_\chi(\lambda_2^i)]=[Z_\chi(\lambda_3^i)]=[L_\chi(\lambda_1^i)]+[L_\chi(\lambda_3^i)]$ and $L_\chi(\lambda_1^i) \ncong L_\chi(\lambda_3^i)$ with $i=3,4$.
\end{corollary}

\end{document}